\documentclass[10pt,twocolumn]{IEEEtran}

\IEEEoverridecommandlockouts                              

\usepackage{amsmath,graphicx,amsfonts,amssymb,amsthm,epsfig,mathrsfs,balance}
\usepackage{subcaption}
\usepackage{tikz}
\usetikzlibrary{shapes}
\usepackage{color}
\usepackage{slashbox}
\usepackage{multirow}
\usepackage{rotating}
\usepackage{algorithm,algpseudocode,algorithmicx}
\usepackage{cite}
\usepackage{url}
\usepackage{framed}
\usepackage{bm}

\newtheorem{assumption}{Assumption}
\newtheorem{theorem}{Theorem}

\newtheorem{definition}{Definition}
\newtheorem{lemma}{Lemma}
\newtheorem{proposition}{Proposition}
\newtheorem{remark}{Remark}

\renewcommand{\det}{{\mathrm{det}}}
\newcommand{\tr}{{\mathrm{trace}}}

\newcommand{\vol}{{\mathrm{vol}}}

\newcommand{\differential}{{\rm{d}}}

\newcommand{\circleline}{\raisebox{0.5pt}{\tikz{\node[draw,scale=0.4,circle,fill=blue](){};\draw[-,blue,solid,line width = 1.0pt](-0.35,0.mm) -- (3.6mm,0.mm)}}}

\newcommand{\triangleline}{\raisebox{0pt}{\tikz{\node[draw,scale=0.3,regular polygon, regular polygon sides=3,fill=black!45!green](){};\draw[-,black!45!green,solid,line width = 1.0pt](-0.35,0.mm) -- (3.6mm,0.mm)}}}

\newcommand{\invtriangleline}{\raisebox{0pt}{\tikz{\node[draw,scale=0.3,regular polygon, regular polygon sides=3,fill=red,rotate=180](){};\draw[-,red,solid,line width = 1.0pt](-0.35,0.mm) -- (3.6mm,0.mm)}}}

\title{\LARGE \bf
Smallest Ellipsoid Containing $p$-Sum of Ellipsoids \\with Application to Reachability Analysis
}

\author{Abhishek Halder
\thanks{Abhishek Halder is with the Department of Applied Mathematics, University of California, Santa Cruz, CA 95064, USA,
        {\tt\small{ahalder@ucsc.edu}}%
}}

\begin{document}

\maketitle
\thispagestyle{empty}
\pagestyle{empty}

\begin{abstract}
We study the problem of ellipsoidal bounding of convex set-valued data, where the convex set is obtained by the $p$-sum of finitely many ellipsoids, for any real $p\geq 1$. The notion of $p$-sum appears in the Brunn-Minkowski-Firey theory in convex analysis, and generalizes several well-known set-valued operations such as the Minkowski sum of the summand convex sets (here, ellipsoids). We derive an outer ellipsoidal parameterization for the $p$-sum of a given set of ellipsoids, and compute the tightest such parameterization for two optimality criteria: minimum trace and minimum volume. For such optimal parameterizations, several known results in the system-control literature are recovered as special cases of our general formula. For the minimum volume criterion, our analysis leads to a fixed point recursion over a scalar that parameterizes the shape matrix of the outer ellipsoid. This recursion is proved to be contractive, and found to converge fast in practice. We apply these results to compute the forward reach sets for a linear control system subject to different convex set-valued uncertainty models for the initial condition and control, generated by varying $p\in[1,\infty]$. Our numerical results show that the proposed fixed point algorithm offers more than two orders of magnitude speed-up in computational time for $p=1$, compared to the existing semidefinite programming approach without significant effect on the numerical accuracy. For $p>1$, the reach set computation results reported here are novel. Our results are expected to be useful in real-time safety critical applications such as decision making for collision avoidance of autonomous vehicles, where the computational time-scale for reach set calculation needs to be much smaller than the vehicular dynamics time-scale.
\end{abstract}




\section{Introduction}
Computing an ellipsoid that contains given set-valued data, is central to many applications such as guaranteeing collision avoidance in robotics \cite{choi2009continuous,best2016real,yan2016path}, robust estimation \cite{schweppe1968recursive,bertsekas1971recursive,schlaepfer1972continuous,reshetnyak1989summation}, system identification \cite{fogel1979system,belforte1990parameter,kosut1992set}, and control \cite{kurzhanski2011optimization,angeli2008ellipsoidal}. To reduce conservatism, one requires such an ellipsoid to be ``smallest" according to some optimality criterion, among all ellipsoids containing the data. Typical examples of optimality criteria are ``minimum volume" and ``minimum sum of the squared semi-axes". A common situation arising in practice is the following: the set-valued data itself is described as set operations (e.g. union, intersection, or Minkowski sum) on other ellipsoids. In this paper, we consider computing the smallest ellipsoid that contains the so-called $p$-sum of finitely many ellipsoids, where $p\in[1,\infty]$. 

As a set operation, the $p$-sum of convex sets returns a new convex set, which loosely speaking, is a combination of the input convex sets. The notion of $p$-sum was introduced by Firey \cite{Firey1962} to generalize the Minkowski sum\footnote[2]{The Minkowski sum of two compact convex sets $\mathcal{X},\mathcal{Y}\subset\mathbb{R}^{d}$ is the set $\mathcal{X} +_{1} \mathcal{Y} := \{\bm{x}+\bm{y} \mid \bm{x} \in \mathcal{X}, \bm{y}\in\mathcal{Y}\}\subset\mathbb{R}^{d}$.} of convex bodies, and was studied in detail by Lutwak \cite{Lutwak1993,Lutwak1996}, who termed the resulting development as \emph{Brunn-Minkowski-Firey theory} (also known as \emph{$L_{p}$ Brunn-Minkowski theory}, see e.g., \cite[Ch. 9.1]{Schneider2014Book}). In this paper, we derive an outer ellipsoidal parameterization that is guaranteed to contain the $p$-sum of given ellipsoids, and then compute the smallest such outer ellipsoid.

Since the $p$-sum subsumes well-known set operations like the Minkowski sum as special case, we recover known results in the systems-control literature \cite{KurzhanskiValyi1997, schweppe1973, MaksarovNorton1996,KuzhanskiVaraiyaBook2014} about the minimum trace and minimum volume outer ellipsoids of such sets, by specializing our optimal parameterization of the $p$-sum of ellipsoids. Furthermore, based on our analytical results, we propose a fixed point recursion to compute the minimum volume outer ellipsoid of the $p$-sum for any real $p\in[1,\infty]$, that has fast convergence. The proposed algorithm not only enables computation for the novel convex uncertainty models (for $p>1$), it also entails orders of magnitude faster runtime compared to the existing semidefinite programming approach for the $p=1$ case. Thus, the contribution of this paper is twofold: (i) generalizing several existing results in the literature on outer ellipsoidal parameterization (Section III) of a convex set obtained as set operations on given ellipsoids, and analyzing its optimality (Section IV); (ii) deriving numerical algorithms (Section V) to compute the minimum volume outer ellipsoid containing the $p$-sum of ellipsoids.

A salient feature of the proposed minimum volume (constrained to a parametric family that we construct) outer ellipsoid algorithm is that it does not require extra parameterized ellipsoids, unlike many other algorithms available in the literature \cite{KuzhanskiVaraiyaBook2014,chernousko2004properties}. Put differently, the proposed algorithms directly process the data of the problem to return the outer approximation without any extra construction.

To illustrate the numerical algorithms derived in this paper, we compute (Section VI) the smallest outer ellipsoidal approximations for the (forward) reach sets of a discrete-time linear control system subject to set-valued uncertainties in its initial conditions and control. When the initial condition and control sets are ellipsoidal, they model weighted norm bounded uncertainties, and at each time, we are led to compute the smallest outer ellipsoid for the Minkowski sum ($p=1$ case). It is found that by specializing the proposed algorithms for $p=1$, we can lower the computational runtime by more than two orders of magnitude compared to the current state-of-the-art, which is to reformulate and solve the same via semidefinite programming. For $p>1$, the initial conditions and controls belong to $p$-sums of ellipsoidal sets, which are convex but not ellipsoidal, in general. In this case too, the proposed algorithms enable computing the smallest outer ellipsoids for the reach sets.

\section{Preliminaries}

\subsubsection{Convex Geometry}
The support function $h_{\mathcal{K}}(\cdot)$ of a compact convex set $\mathcal{K}\subset\mathbb{R}^{d}$, is 
\begin{align}
	h_{\mathcal{K}}(\bm{y}) := \sup \big\{\langle \bm{x}, \bm{y} \rangle \:\mid\: \bm{x}\in\mathcal{K}, \: \bm{y}\in\mathbb{R}^{d}\big\},
	\label{SptFnDef}
\end{align}
where $\langle \cdot, \cdot \rangle$ denotes the standard Euclidean inner product. The function $h_{\mathcal{K}} : \mathcal{K} \mapsto \mathbb{R}$, and can be viewed geometrically as the (signed) perpendicular distance from the origin to the supporting hyperplane of $\mathcal{K}$, which has the outer normal vector $\bm{y}$. Thus, the support function returns negative value if and only if the normal vector $\bm{y}$ points into the open halfspace containing the origin. The support function $h_{\mathcal{K}}(\cdot)$ can also be seen as (see e.g., \cite[Theorem 13.2]{RockafellerBook1970}) the Legendre-Fenchel conjugate of the indicator function of the set $\mathcal{K}$, and thus uniquely determines the set $\mathcal{K}$.

The following properties of the support function are well-known: \\
(i) convexity: $h_{\mathcal{K}}(\bm{y})$ is a convex function in $\bm{y}\in\mathbb{R}^{d}$,\\
(ii) positive homogeneity: $h_{\mathcal{K}}(\alpha\bm{y}) = \alpha h_{\mathcal{K}}(\bm{y})$ for $\alpha > 0$,\\
(iii) sub-additivity: $h_{\mathcal{K}}(\bm{y}+\bm{z}) \leq h_{\mathcal{K}}(\bm{y}) + h_{\mathcal{K}}(\bm{z})$ for $\bm{y},\bm{z}\in\mathbb{R}^{d}$,\\
(iv) inclusion: given compact convex sets $\mathcal{K}_{1}$ and $\mathcal{K}_{2}$, the inclusion  $\mathcal{K}_{1} \subseteq \mathcal{K}_{2}$ holds if and only if $h_{\mathcal{K}_{1}}(\bm{y}) \leq h_{\mathcal{K}_{2}}(\bm{y})$ for all $\bm{y}\in\mathbb{R}^{d}$,\\
(v) affine transformation: $h_{\bm{A}\mathcal{K} + \bm{b}}(\bm{y}) = h_{\mathcal{K}}(\bm{A}^{\top}\bm{y}) + \langle \bm{b}, \bm{y} \rangle$, for $\bm{A}\in\mathbb{R}^{d\times d}, \bm{b}\in\mathbb{R}^{d}, \bm{y}\in\mathbb{R}^{d}$.
\begin{definition}\label{pSumDefn}
($p$-\textbf{Sum of convex sets})	Given compact convex sets $\mathcal{K}_{1},\mathcal{K}_{2}\subset \mathbb{R}^{d}$, their $p$-sum \cite{Firey1962} is a new compact convex set $\mathcal{K}\subset \mathbb{R}^{d}$ defined via its support function  
\begin{align}
	h_{\mathcal{K}}(\bm{y}) = \left(h_{\mathcal{K}_{1}}^{p}(\bm{y}) \: + \: h_{\mathcal{K}_{2}}^{p}(\bm{y})\right)^{\frac{1}{p}}, \quad 1 \leq p \leq \infty,
	\label{FireypSum}
\end{align}
and we write $\mathcal{K} := \mathcal{K}_{1} +_{p} \mathcal{K}_{2}$.
\end{definition}
\begin{figure}[t]
 \centering
\includegraphics[width=\linewidth]{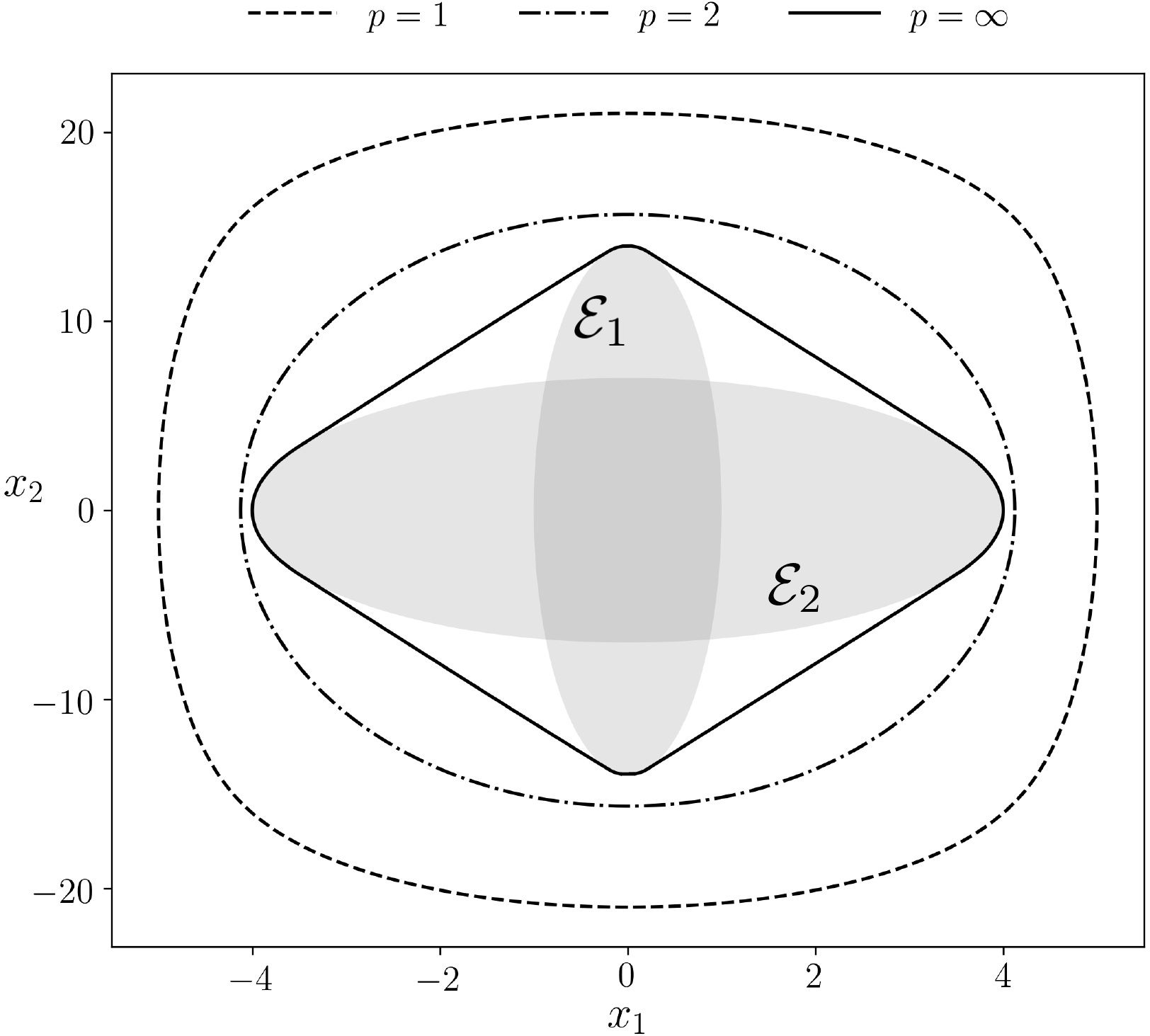}
 \caption{The $p$-sum of two input ellipses $\mathcal{E}_{1}$ and $\mathcal{E}_{2}$ (\emph{filled gray}) are shown for $p=1,2,\infty$. For varying $p \in [1,\infty]$, the $p$-sum $\mathcal{E}_{1} +_{p}\mathcal{E}_{2}$ defines a nested sequence of convex sets -- the outermost being the Minkowski sum ($p=1$, \emph{dashed boundary}), and the innermost being the convex hull of the union ($p=\infty$, \emph{solid boundary}) of $\mathcal{E}_{1}$ and $\mathcal{E}_{2}$. In this case, the $p$-sum is an ellipse only for $p=2$ (\emph{dash-dotted boundary}).}
 \label{pSumOf2Ellipses}
\end{figure}
Special cases of the $p$-sum are encountered frequently in practice. For example, when $p=1$, the set $\mathcal{K}=\mathcal{K}_{1} +_{1} \mathcal{K}_{2}$ is the Minkowski sum of $\mathcal{K}_{1}$ and $\mathcal{K}_{2}$, and \[h_{\mathcal{K}_{1} +_{1} \mathcal{K}_{2}}(\cdot) = h_{\mathcal{K}_{1}}(\cdot) + h_{\mathcal{K}_{2}}(\cdot).\] When $p=\infty$, the set $\mathcal{K}=\mathcal{K}_{1} +_{\infty} \mathcal{K}_{2}$ is the convex hull of the union of $\mathcal{K}_{1}$ and $\mathcal{K}_{2}$, and \[h_{\mathcal{K}_{1} +_{\infty} \mathcal{K}_{2}}(\cdot) = \max\big\{h_{\mathcal{K}_{1}}(\cdot), h_{\mathcal{K}_{2}}(\cdot)\big\}.\] For $1 \leq p < q \leq \infty$, we have the inclusion \cite[p. 20]{Firey1962} (See Fig. \ref{pSumOf2Ellipses}):
\begin{eqnarray}
\mathcal{K}_{1} \cup \mathcal{K}_{2} \subseteq \mathcal{K}_{1} +_{\infty} \mathcal{K}_{2} \subseteq \hdots &\subseteq \mathcal{K}_{1} +_{q} \mathcal{K}_{2} \subseteq \mathcal{K}_{1} +_{p} \mathcal{K}_{2} \nonumber\\
&\subseteq \hdots \subseteq \mathcal{K}_{1} +_{1} \mathcal{K}_{2},
\label{pSumInclusion}	
\end{eqnarray}
which follows (Appendix \ref{AppendixSetInclusionpSum}) from the support function inequality 
\[h_{\mathcal{K}_{1} +_{q} \mathcal{K}_{2}}(\cdot) \leq h_{\mathcal{K}_{1} +_{p} \mathcal{K}_{2}}(\cdot), \quad 1 \leq p < q \leq \infty.\]
From Definition \ref{pSumDefn}, it is easy to see that the $p$-sum is commutative and associative, that is,
\begin{subequations}\label{CommuAsso}
\begin{align}
\label{Commu}
	\mathcal{K}_{1} +_{p} \mathcal{K}_{2} &= \mathcal{K}_{2} +_{p} \mathcal{K}_{1},\\
\label{Asso}	
\left(\mathcal{K}_{1} +_{p} \mathcal{K}_{2}\right) +_{p} \mathcal{K}_{3} &= \mathcal{K}_{1} +_{p} \left(\mathcal{K}_{2} +_{p} \mathcal{K}_{3}\right),
\end{align}
\end{subequations}
where the compact sets $\mathcal{K}_{1}, \mathcal{K}_{2}, \mathcal{K}_{3}$ are convex. Furthermore, linear transformation is distributive over $p$-sum, i.e., 
\begin{eqnarray}
\label{DistributiveLinTrans}
\bm{A}\left(\mathcal{K}_{1} +_{p} \mathcal{K}_{2}\right) = \bm{A}\mathcal{K}_{1} +_{p} \bm{A}\mathcal{K}_{2}, \quad \bm{A}\in\mathbb{R}^{d\times d},	
\end{eqnarray}
which is immediate from the aforesaid property (v) and equation (\ref{FireypSum}). We mention here that it is often convenient to express $h_{\mathcal{K}}(\cdot)$ as function of the unit vector $\bm{y}/\parallel \bm{y}\parallel_{2}$ in $\mathbb{R}^{d}$ (see Fig. \ref{SptFnOfpSumOf2Ellipses}).

\begin{remark}
At first glance, it might seem odd that the Minkowski sum is defined pointwise as $\mathcal{X} +_{1} \mathcal{Y} := \{\bm{x}+\bm{y} \mid \bm{x} \in \mathcal{X}, \bm{y}\in\mathcal{Y}\}$, which remains well-defined for $\mathcal{X},\mathcal{Y}$ compact (not necessarily convex), but the $p$-sum in Definition \ref{pSumDefn} is given via support functions and requires the summand sets to be convex. Indeed, a pointwise definition for the $p$-sum was proposed in \cite{lutwak2012brunn} for compact summand sets $\mathcal{X}$ and $\mathcal{Y}$, given by $\mathcal{X} +_{p} \mathcal{Y} := \{(1-\mu)^{1/p^{\prime}}\bm{x}+\mu^{1/p^{\prime}}\bm{y} \mid \bm{x} \in \mathcal{X}, \bm{y}\in\mathcal{Y}, 0\leq\mu\leq 1\}$, where $p^{\prime}$ denotes the H\"{o}lder conjugate of $p$, i.e., $\frac{1}{p}+\frac{1}{p^{\prime}}=1$. This pointwise definition was shown to reduce to Definition \ref{pSumDefn} provided the compact summands $\mathcal{X},\mathcal{Y}$ are also convex. 
\end{remark}

\subsubsection{Ellipsoids}
Let $\mathbb{S}_{+}^{d}$ be the cone of $d\times d$ symmetric positive definite matrices. For an ellipsoid with center $\bm{q}\in\mathbb{R}^{d}$ and shape matrix $\bm{Q} \in \mathbb{S}_{+}^{d}$, denoted by \[\mathcal{E}\left(\bm{q}, \bm{Q}\right) := \{\bm{x} \in \mathbb{R}^{d} : \left(\bm{x} - \bm{q}\right)^{\top}\bm{Q}^{-1}\left(\bm{x} - \bm{q}\right) \leq 1\},\]
(\ref{SptFnDef}) reduces to
\begin{align}
	h_{\mathcal{E}\left(\bm{q}, \bm{Q}\right)}(\bm{y}) = \langle \bm{q}, \bm{y} \rangle \: + \: \sqrt{\langle\bm{Q}\bm{y}, \bm{y}\rangle}. 
\label{SptFnEllipsoid}	
\end{align}
Furthermore, the volume of the ellipsoid $\mathcal{E}\left(\bm{q}, \bm{Q}\right)$ is given by 
\begin{eqnarray}
\vol\left(\mathcal{E}\left(\bm{q},\bm{Q}\right)\right) = \frac{\vol\left(\mathcal{B}_{1}^{d}\right)}{\sqrt{\det\left(\bm{Q}^{-1}\right)}} = \frac{\pi^{\frac{d}{2}}}{\Gamma\left(\frac{d}{2} + 1\right)}\sqrt{\det\left(\bm{Q}\right)},
\label{VolEllipsoid}	
\end{eqnarray}
where $\mathcal{B}_{1}^{d}$ denotes the $d$-dimensional unit ball, and $\Gamma(\cdot)$ denotes the Gamma function, i.e., $\Gamma(s) := \int_{0}^{\infty}\eta^{s-1}\exp\left(-\eta\right)\differential\eta$, provided the real part of $s>0$.

\begin{figure}[t]
 \centering
\includegraphics[width=\linewidth]{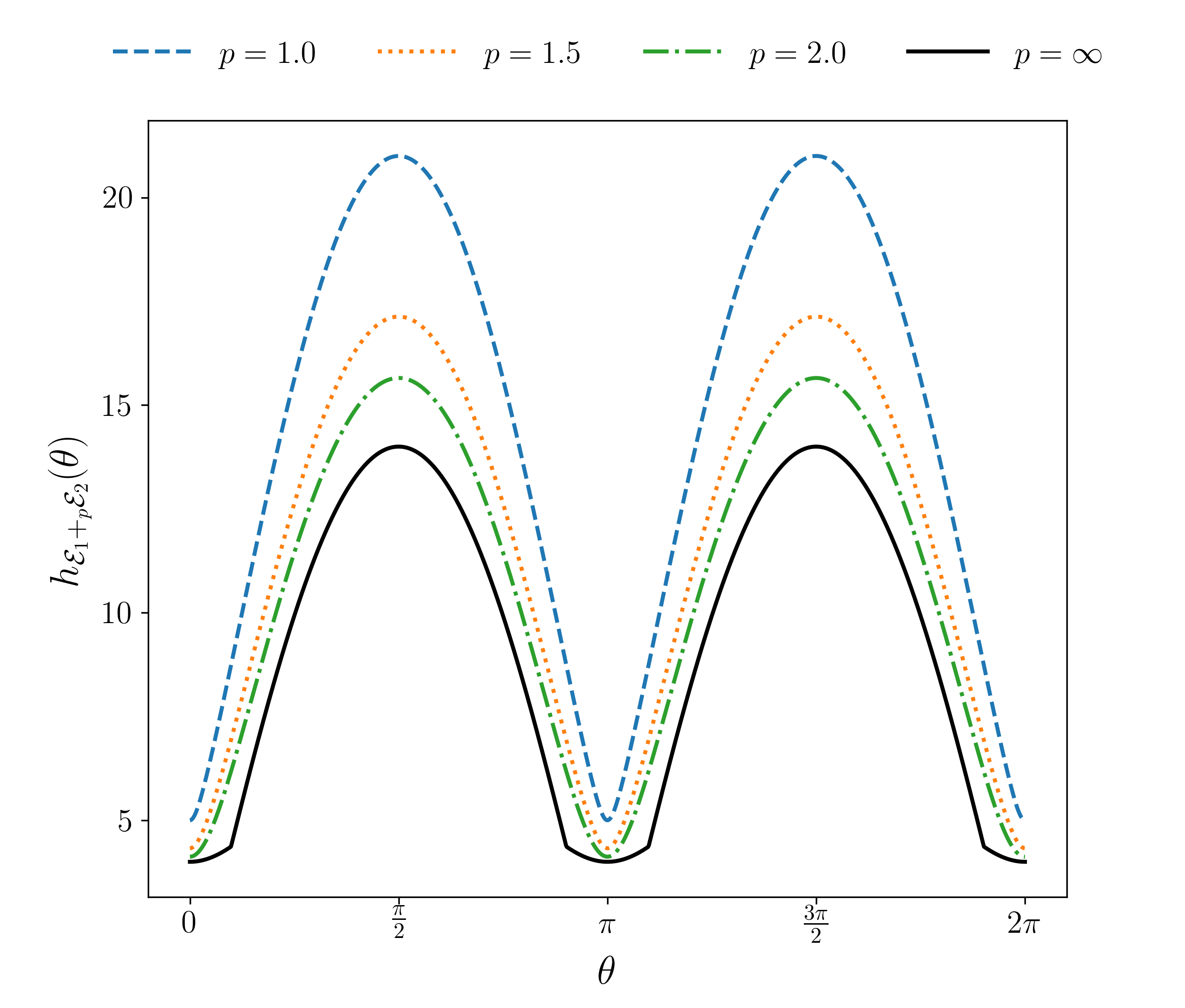}
 \caption{The support function $h_{\mathcal{E}_{1}+_{p} \mathcal{E}_{2}}(\cdot)$ for the $p$-sum of two input ellipses $\mathcal{E}_{1}$ and $\mathcal{E}_{2}$ are shown for different values of $p\geq 1$, where $\mathcal{E}_{i}:=\big\{(x,y)\in\mathbb{R}^{2} \mid \frac{x^{2}}{a_{i}^{2}}+\frac{y^{2}}{b_{i}^{2}}\leq 1\big\}$, $i=1,2$, with $a_{1}=4,b_{1}=7,a_{2}=1,b_{2}=14$. The corresponding $p$-sum sets for $p=1,2,\infty$ are as in Fig. 1. In this case, $h_{\mathcal{E}_{1}+_{p} \mathcal{E}_{2}}(\theta) = \left(\sum_{i=1,2}(a_{i}^{2}\cos^{2}\theta + b_{i}^{2}\sin^{2}\theta)^{p/2}\right)^{1/p}$ for $p\in[1,\infty)$, and $h_{\mathcal{E}_{1}+_{\infty} \mathcal{E}_{2}}(\theta)=\max_{i=1,2}(a_{i}^{2}\cos^{2}\theta + b_{i}^{2}\sin^{2}\theta)^{1/2}$, where $\theta\in[0,2\pi)$.}
 \label{SptFnOfpSumOf2Ellipses}
\end{figure}

An alternative ellipsoidal parameterization can be obtained via a matrix-vector-scalar triple $(\bm{A},\bm{b},c)$ encoding the quadratic form, i.e., $\mathcal{E}(\bm{A},\bm{b},c) := \{\bm{x} \in \mathbb{R}^{d} : \bm{x}^{\top}\bm{A}\bm{x} + 2\bm{x}^{\top}\bm{b} + c \leq 0\}$. The following relations among $(\bm{A},\bm{b},c)$ and $(\bm{q},\bm{Q})$ parameterizations for ellipsoid will be useful in the later part of this paper:
\begin{subequations}\label{AbcandQq}
\begin{align}
\label{Qq2Abc}
	&\bm{A} = \bm{Q}^{-1},\quad \bm{b} = -\bm{Q}^{-1}\bm{q},\quad c = \bm{q}^{\top}\bm{Q}^{-1}\bm{q} - 1,\\
\label{Ab2Qq}		
&\bm{Q} = \bm{A}^{-1}, \quad \bm{q} = -\bm{Q}\bm{b}.
\end{align}
\end{subequations}

\section{Parameterized Outer Ellipsoid}
Given two ellipsoids $\mathcal{E}(\bm{q}_{1},\bm{Q}_{1}),\mathcal{E}(\bm{q}_{2},\bm{Q}_{2})\subset \mathbb{R}^{d}$, consider their $p$-sum
\begin{alignat}{2}
	\mathcal{E}(\bm{q}_{1},\bm{Q}_{1}) +_{p} \mathcal{E}(\bm{q}_{2},\bm{Q}_{2}), \quad 1\leq p \leq \infty,
	\label{pSumEll}
\end{alignat}
which is convex but not an ellipsoid in general (see Fig. \ref{pSumOf2Ellipses} and \ref{SptFnOfpSumOf2Ellipses}). We want to determine an ellipsoid $\mathcal{E}(\bm{q},\bm{Q})\subset\mathbb{R}^{d}$, as function of the input ellipsoids, that is guaranteed to contain the $p$-sum (\ref{pSumEll}). For this to happen, we must have (from (\ref{FireypSum}) and property (iv) in Section II.1)
	\begin{alignat}{2}
		&h_{\mathcal{E}\left(\bm{q}, \bm{Q}\right)}^{p}(\bm{y}) \geq h_{\mathcal{E}\left(\bm{q}_{1}, \bm{Q}_{1}\right)}^{p}(\bm{y}) + h_{\mathcal{E}\left(\bm{q}_{2}, \bm{Q}_{2}\right)}^{p}(\bm{y}), 
\nonumber\\
\vspace*{0.01in}\nonumber\\
\stackrel{(\ref{SptFnEllipsoid})}{\Rightarrow} &\left(\bm{y}^{\top}\bm{q} +\!\sqrt{\bm{y}^{\top}\bm{Q}\bm{y}}\right)^{p} \geq \displaystyle\sum_{i=1,2}\!\!\left(\bm{y}^{\top}\bm{q}_{i}+\!\sqrt{\bm{y}^{\top}\bm{Q}_{i}\bm{y}}\right)^{p},
\label{SptFnEllIneq}
	\end{alignat}
for all $\bm{y}\in\mathbb{R}^{d}$. In the rest of this paper, we make the following assumption.

\begin{assumption}\label{AssumpCenterZero}
	The center vectors for the summand ellipsoids in a $p$-sum are assumed to be zero, i.e., $\bm{q}_{1}=\bm{q}_{2}=\bm{0}$.
\end{assumption}

Under Assumption (\ref{AssumpCenterZero}), and $p=2$, it follows from (\ref{SptFnEllipsoid}) and (\ref{SptFnEllIneq}) that (\ref{pSumEll}) is an ellipsoid $\mathcal{E}\left(\bm{0},\bm{Q}_{1}+\bm{Q}_{2}\right)$. 

For $p\neq 2$, the convex set (\ref{pSumEll}) is again not an ellipsoid in general, but we can parameterize an outer ellipsoid $\mathcal{E}\left(\bm{0},\bm{Q}\right)\supseteq\mathcal{E}(\bm{0},\bm{Q}_{1}) +_{p} \mathcal{E}(\bm{0},\bm{Q}_{2})$ as follows. 

For $\alpha,\gamma > 1$, let 
\begin{eqnarray}
\bm{Q} := \alpha\bm{Q}_{1} + \gamma\bm{Q}_{2}, \; \text{and} \; g_{i}^{2} := \bm{y}^{\top}\bm{Q}_{i}\bm{y} \geq 0, \quad i=1,2.
\label{Qalphagamma}	
\end{eqnarray}
For the time being, we think of $\alpha,\gamma > 1$ as free parameters. We will see that it is possible to re-parameterize $\bm{Q}$ in terms of a single parameter $\beta >0$, by expressing both $\alpha$ and $\gamma$ as appropriate functions of $\beta$, while guaranteeing the inclusion of the $p$-sum in $\mathcal{E}\left(\bm{0},\bm{Q}\right)$.

With the standing assumptions $\bm{q}_{1}=\bm{q}_{2}=\bm{0}$, we can re-write (\ref{SptFnEllIneq}) as 
\begin{eqnarray}
\left(\alpha g_{1}^{2} + \gamma g_{2}^{2}\right)^{\frac{p}{2}}	&\:\geq\: &g_{1}^{p} + g_{2}^{p},\nonumber\\
\Rightarrow \left(\alpha g_{1}^{2} + \gamma g_{2}^{2}\right)^{p}	&\:\geq\:& g_{1}^{2p} + g_{2}^{2p} + 2 g_{1}^{p} g_{2}^{p}.
\label{g1g2inequality}
\end{eqnarray}
To proceed further, we need the following Lemma.
\begin{lemma}
A convex function $f(\cdot)$ with $f(0)=0$ is super-additive on $[0,\infty)$, i.e., $f(x+y) \geq f(x) + f(y)$ for all $x,y\geq 0$.
\end{lemma}
\begin{proof}
For $0 \leq \lambda \leq 1$, by convexity
\[f(\lambda x) = f(\lambda x + (1-\lambda)0) \leq \lambda f(x) + (1-\lambda)f(0) = \lambda f(x).\]	
Therefore, we get
\[f(x) = f\left(\frac{x}{x+y}(x+y)\right) \leq \frac{x}{x+y}f(x+y),\]
and \[f(y) = f\left(\frac{y}{x+y}(x+y)\right) \leq \frac{y}{x+y}f(x+y).\]
Adding the last two inequalities, $f(x) + f(y) \leq f(x+y)$.
\end{proof}
In Theorem 1 below, we use Lemma 1 to derive an explicit parameterization of the shape matrix $\bm{Q}(\beta) \in \mathbb{S}^{d}_{+}$, $\beta>0$, that guarantees an outer ellipsoidal containment of the $p$-sum \[\mathcal{E}(\bm{0},\bm{Q}_{1}) +_{p} \mathcal{E}(\bm{0},\bm{Q}_{2}).\]
\begin{theorem}
Given a scalar $\beta > 0$, and a pair of matrices $\bm{Q}_{1},\bm{Q}_{2} \in \mathbb{S}_{+}^{d}$, let
\begin{eqnarray}
	\bm{Q}(\beta) = \left(\!\!1 + \displaystyle\frac{1}{\beta} \!\right)^{\!\!\frac{1}{p}}\!\!\bm{Q}_{1} + \left(1 + \beta\right)^{\frac{1}{p}}\bm{Q}_{2}, \: p\in[1,\infty)\setminus\{2\}.
\label{Qbetap}	
\end{eqnarray}
Then $\mathcal{E}\left(\bm{0},\bm{Q}(\beta)\right)\supseteq\mathcal{E}(\bm{0},\bm{Q}_{1}) +_{p} \mathcal{E}(\bm{0},\bm{Q}_{2})$.	
\end{theorem}
\begin{proof}
Since $f(x) := x^{p}$ is convex on $[0,\infty)$ for $1\leq p < \infty$, Lemma 1 yields
	\[\left(\alpha g_{1}^{2} + \gamma g_{2}^{2}\right)^{p} \:\geq\: \alpha^{p} g_{1}^{2p} + \gamma^{p} g_{2}^{2p},\]
or equivalently, 
\begin{eqnarray}
\left(\alpha g_{1}^{2} + \gamma g_{2}^{2}\right)^{p} = \alpha^{p} g_{1}^{2p} + \gamma^{p} g_{2}^{2p} + \xi, \; \text{where}\;\xi\geq0.
\label{xiEquality}
\end{eqnarray} 
Combining (\ref{g1g2inequality}) and (\ref{xiEquality}), we obtain
\begin{eqnarray}
\left(\alpha^{p}-1\right)\left(g_{1}^{p}\right)^{2} + \left(\gamma^{p}-1\right)\left(g_{2}^{p}\right)^{2} - 2 g_{1}^{p}g_{2}^{p} + \xi \geq 0.
\label{IntermedIneq}	
\end{eqnarray}
Since $\alpha > 1$, we have $\alpha^{p} > 1$ for $p\geq 1$. Therefore, multiplying both sides of (\ref{IntermedIneq}) by $(\alpha^{p} - 1)>0$, and then adding and subtracting $(g_{2}^{p})^{2}$, we get\begin{eqnarray}
	\left(\left(\alpha^{p}-1\right)g_{1}^{p} - g_{2}^{p}\right)^{2} + \left(\left(\alpha^{p}-1\right)\left(\gamma^{p}-1\right) - 1\right)(g_{2}^{p})^{2} \nonumber\\
	 + \left(\alpha^{p} - 1\right)\xi \geq 0.
\label{SumOfSquaresIneq}	 
\end{eqnarray}
A sufficient condition to satisfy the inequality (\ref{SumOfSquaresIneq}) is to choose $\alpha,\gamma>1$ such that
\[\left(\alpha^{p}-1\right)\left(\gamma^{p}-1\right) \geq 1.\]
Letting $\alpha^{p}-1 := \beta^{-1}$ and $\gamma^{p}-1 = \beta$, $\beta>0$, and using (\ref{Qalphagamma}), we arrive at (\ref{Qbetap}).
\end{proof}
\begin{remark}\label{ParamGeneralizationRemark}
	The parameterization (\ref{Qbetap}) generalizes the outer ellipsoidal parameterization containing the Minkowski sum ($p=1$ case), well-known in the systems-control literature (see e.g., \cite[p. 104]{KurzhanskiValyi1997},\cite{schweppe1973, MaksarovNorton1996}). For this special case $p=1$, a discussion about equivalent parameterizations can be found in \cite[Section II.A]{HalderMinkSum2018}.
\end{remark}

\section{Optimal Parameterization}\label{OptParaSection}
To reduce conservatism, it is desired that the parameterized outer ellipsoid $\mathcal{E}\left(\bm{0},\bm{Q}(\beta)\right)$ in Theorem 1 containing the $p$-sum, be as tight as possible. One way to promote ``tightness" is by minimizing the sum of the squared semi-axes lengths of $\mathcal{E}\left(\bm{0},\bm{Q}(\beta)\right)$, which amounts to minimizing $\tr\left(\bm{Q}(\beta)\right)$ over $\beta > 0$. Another possible way to promote ``tightness" is by minimizing the volume of $\mathcal{E}\left(\bm{0},\bm{Q}(\beta)\right)$, which, thanks to (\ref{VolEllipsoid}), amounts to minimizing $\log\:\det\left(\bm{Q}(\beta)\right)$. We next analyze these optimality criteria. A discussion on different optimality criteria promoting different notions of ``ellipsoidal tightness" can be found in \cite[p. 226]{chernousko2004properties}. 

\subsection{Minimum Trace Outer Ellipsoid}
We consider the optimization problem
\begin{eqnarray}
\underset{\beta > 0}{\text{minimize}}\quad\tr\left(\bm{Q}(\beta)\right),	
\end{eqnarray}
where $\bm{Q}(\beta)$ is given by (\ref{Qbetap}), and let \[\beta^{*}_{{\rm{tr}}} := \underset{\beta > 0}{\arg\min}\:\tr\left(\bm{Q}(\beta)\right).\]

Setting $\displaystyle\frac{\partial}{\partial\beta}\tr\left(\bm{Q}(\beta)\right) = 0$, and using the linearity of trace operator, straightforward calculation yields
\begin{eqnarray}
\beta^{*}_{{\rm{tr}}} = \left(\displaystyle\frac{\tr\left(\bm{Q}_{1}\right)}{\tr\left(\bm{Q}_{2}\right)}\right)^{\frac{p}{1+p}}, \quad p\in[1,\infty)\setminus\{2\},
\label{betamintr}	
\end{eqnarray}
and 
\begin{eqnarray*}
&&\frac{\partial^{2}}{\partial\beta^{2}}\tr\left(\bm{Q}(\beta)\right)\Big\vert_{\beta=\beta^{*}_{{\rm{tr}}}} \\
&=& \frac{1}{p}\left(\frac{1}{p}+1\right)\left(\beta^{*}_{{\rm{tr}}}\right)^{-\frac{1}{p}-2}\left(\beta^{*}_{{\rm{tr}}} + 1\right)^{\frac{1}{p}-1}\tr\left(\bm{Q}_{1}\right) > 0.	
\end{eqnarray*}
The formula (\ref{betamintr}) generalizes the previously known formula for $p=1$ case (minimum trace ellipsoid containing the Minkowski sum of two given ellipsoids) reported in \cite[Appendix A.2]{MaksarovNorton1996} and in \cite[Lemma 2.5.2(a)]{KurzhanskiValyi1997}.

\subsection{Minimum Volume Outer Ellipsoid}
Now we consider the optimization problem
\begin{eqnarray}
\underset{\beta > 0}{\text{minimize}}\quad\log\:\det\left(\bm{Q}(\beta)\right),	
\end{eqnarray}
where $\bm{Q}(\beta)$ is given by (\ref{Qbetap}), and let 
\begin{eqnarray}
	\beta^{*}_{{\rm{vol}}} := \underset{\beta > 0}{\arg\min}\:\log\:\det\left(\bm{Q}(\beta)\right).
	\label{minVolargmin}
\end{eqnarray}

To simplify the first order condition of optimality $\frac{\partial}{\partial\beta}\log\det\left(\bm{Q}(\beta)\right)=0$, we notice that the matrix $\bm{R}:= \bm{Q}_{1}^{-1}\bm{Q}_{2}$ is diagonalizable \cite[Section III.A, Lemma 1]{HalderMinkSum2018}, and denote its spectral decomposition as $\bm{R}:=\bm{S}\bm{\Lambda}\bm{S}^{-1}$. Further, let the eigenvalues of $\bm{R}$ be $\{\lambda_{i}\}_{i=1}^{d}$, which are all positive (see the discussion following Proposition 1 in \cite{HalderMinkSum2018}). Then direct calculation gives
\begin{eqnarray}
\displaystyle\frac{\partial}{\partial\beta}\log\det\left(\bm{Q}(\beta)\right) = \tr\left((\bm{Q}(\beta))^{-1}\displaystyle\frac{\partial}{\partial\beta}\bm{Q}(\beta)\right)\nonumber\\
= -\displaystyle\frac{1}{p\beta(1+\beta)}\:\tr\left(\left(\bm{I} + \beta^{\frac{1}{p}}\bm{R}\right)^{-1} \left(\bm{I} - \beta^{3-\frac{1}{p}}\bm{R}\right)\right)\nonumber\\
= -\displaystyle\frac{1}{p\beta(1+\beta)}\:\tr\left(\left(\bm{I} + \beta^{\frac{1}{p}}\bm{\Lambda}\right)^{-1} \left(\bm{I} - \beta^{3-\frac{1}{p}}\bm{\Lambda}\right)\right),
\label{FOOC}		
\end{eqnarray}
wherein the last step follows from substituting $\bm{I}=\bm{S}\bm{S}^{-1}$, $\bm{R}=\bm{S}\bm{\Lambda}\bm{S}^{-1}$, and using the invariance of trace of matrix product under cyclic permutation. 

Therefore, from (\ref{FOOC}), the first order optimality condition $\frac{\partial}{\partial\beta}\log\det\left(\bm{Q}(\beta)\right)=0$ is equivalent to the following nonlinear algebraic equation:
\begin{eqnarray}
\displaystyle\sum_{i=1}^{d} \displaystyle\frac{1 - \beta^{3-\frac{1}{p}}\lambda_{i}}{1 + \beta^{\frac{1}{p}}\lambda_{i}} = 0, \quad p\in[1,\infty)\setminus\{2\},
\label{DerivativeEqualsZero}	
\end{eqnarray}
to be solved for $\beta>0$, with known parameters $\lambda_{i}>0$, $i=1,\hdots,d$. If (\ref{DerivativeEqualsZero}) admits unique positive root (which seems non-obvious, and will be proved next), then it would indeed correspond to the argmin in (\ref{minVolargmin}) since
\begin{align}
	&\frac{\partial^{2}}{\partial\beta^{2}}\log\det\left(\bm{Q}(\beta)\right)\bigg\rvert_{\beta>0} \nonumber\\
= \frac{1}{p\beta(1+\beta)} &\displaystyle\sum_{i=1}^{d} \frac{(2-\frac{1}{p})\beta^{2}\lambda_{i}^{2} + \left[(3-\frac{1}{p})\beta^{2-\frac{1}{p}}+ \frac{1}{p}\beta^{\frac{1}{p}-1}\right]\lambda_{i}}{(1 + \beta\lambda_{i})^{2}}\nonumber\\
>0, &\qquad \text{for} \quad p \in[1,\infty)\setminus\{2\}. 
\label{SecondDerivativePositive}	
\end{align}
In the following Theorem, we establish the uniqueness of the positive root. 
\begin{theorem}
Given $\lambda_{i}>0$, $i=1,\hdots,d$, equation (\ref{DerivativeEqualsZero}) in variable $\beta$ admits unique positive root.	
\end{theorem} 
\begin{proof}
We start by rewriting (\ref{DerivativeEqualsZero}) as
\begin{eqnarray}
\displaystyle\sum_{i=1}^{d}\left(\beta^{3-\frac{1}{p}}-\frac{1}{\lambda_{i}}\right)\displaystyle\prod_{\stackrel{j=1}{j\neq i}}^{d}\left(\beta^{\frac{1}{p}} + \frac{1}{\lambda_{j}}\right) = 0.
\label{AfterDivisionBylambdaprod}	
\end{eqnarray}
Now let $e_{r} \equiv e_{r}\left(\frac{1}{\lambda_{1}}, \hdots, \frac{1}{\lambda_{d}}\right)$ for $r=1,\hdots,d$, denote the $r$\textsuperscript{th} elementary symmetric polynomial \cite[Ch. 2.22]{HardyLittlewoodPolyaBook} in variables 	$\frac{1}{\lambda_{1}}, \hdots, \frac{1}{\lambda_{d}}$, that is,
\[e_{r} \equiv e_{r}\left(\frac{1}{\lambda_{1}}, \hdots, \frac{1}{\lambda_{d}}\right) := \displaystyle\sum_{1\leq i_{1} < i_{2} \hdots < i_{r} \leq d} \displaystyle\frac{1}{\lambda_{i_{1}}\hdots\lambda_{i_{r}}}.\]
For example, 
\[e_{1} =\!\!\displaystyle\sum_{1\leq i\leq d}\!\!\lambda_{i}^{-1}, \quad e_{2} =\!\!\displaystyle\sum_{1\leq i<j\leq d}\!\!\!\!\!\left(\lambda_{i}\lambda_{j}\right)^{-1}, \quad e_{d} = \!\!\left(\displaystyle\prod_{1\leq i\leq d}\!\!\lambda_{i}\right)^{\!\!-1},\]
and $e_{0}=1$ by convention. For all $r=1,\hdots,d$, we have $e_{r}>0$ since $\lambda_{1}, \hdots, \lambda_{d}$ are all positive. Letting $\zeta := \beta^{\frac{1}{p}}$, we write (\ref{AfterDivisionBylambdaprod}) in the expanded form
\begin{eqnarray}
\zeta^{3p-1}\bigg\{\displaystyle\sum_{r=1}^{d}\left(d-r+1\right)e_{r-1}\zeta^{d-r}\bigg\} \: - \: \displaystyle\sum_{r=1}^{d}re_{r}\zeta^{d-r} = 0.
\label{ExpandedFormInxi}	
\end{eqnarray}
We notice that (\ref{ExpandedFormInxi}) is a polynomial in $\zeta$, in which the coefficients undergo exactly one change in sign. Therefore, by Descartes' rule of sign, the equation (\ref{ExpandedFormInxi}) (equivalently (\ref{AfterDivisionBylambdaprod}) or (\ref{DerivativeEqualsZero})) admits unique positive root.
\end{proof}
Next, we give an algorithm to compute the unique positive root $\beta^{*}_{{\rm{vol}}}$ of the equation (\ref{DerivativeEqualsZero}), and show how the same can be used for reachability analysis for linear control systems.


\begin{figure}[t]
 \centering
\includegraphics[width=\linewidth]{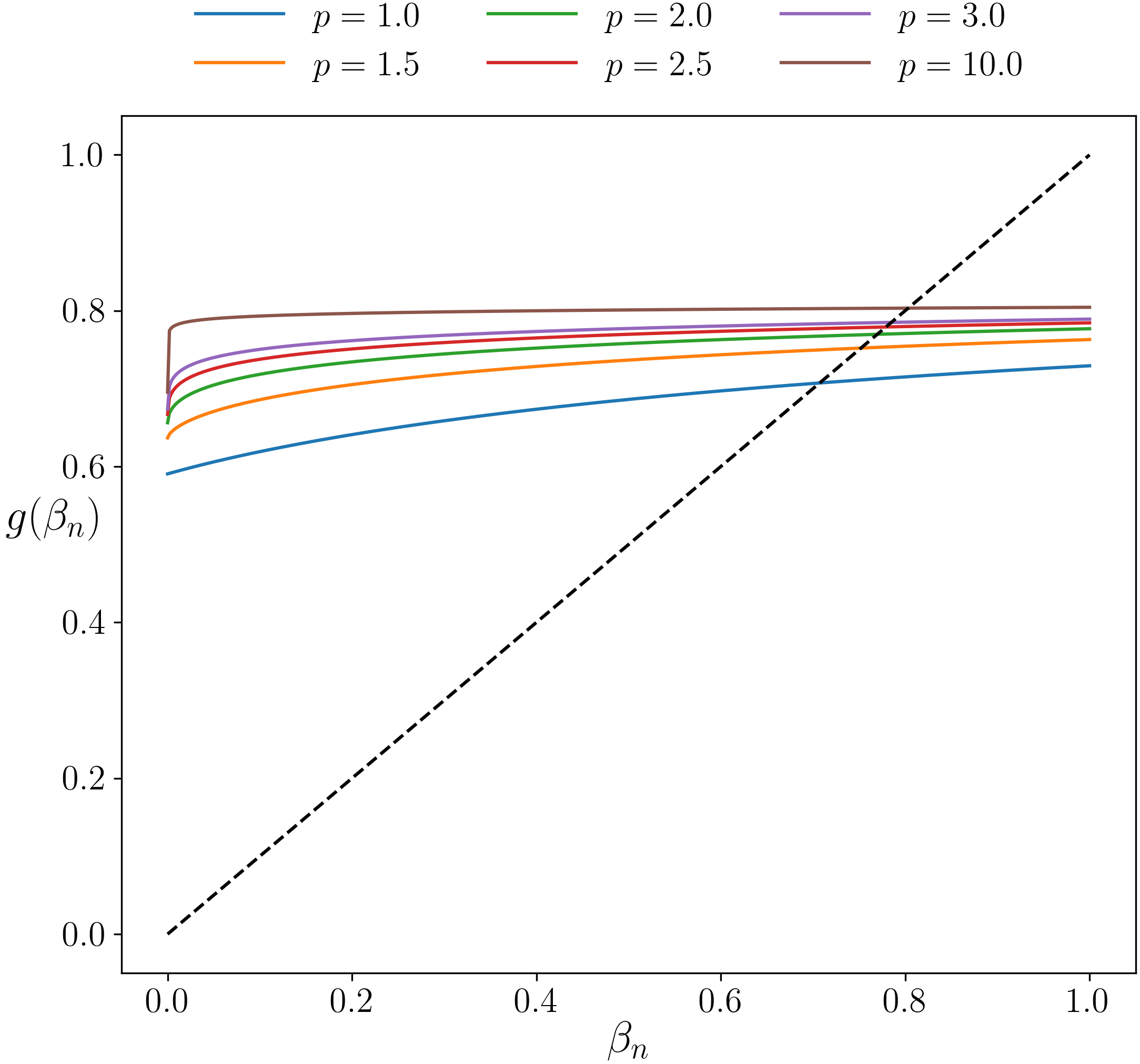}
 \caption{The unique positive fixed point $\beta^{*}_{\text{vol}}$ for the map $\beta \mapsto g(\beta)$ given by (\ref{FPE}), is shown for $p=1,1.5,2,2.5,3,10$, all with $d=3$, and with parameters $\{\lambda_{1},\lambda_{2},\lambda_{3}\}=\{5, 0.6, 3\}$. For any given $p\geq 1$, the fixed point $\beta^{*}_{\text{vol}}$ is the point of intersection between the dashed straight line and the corresponding solid curve $g(\beta)$.}
 \label{FixedPtFn}
\end{figure}

\section{Algorithms and Applications}

\subsection{Computing the Minimum Volume Outer Ellipsoid}
Thanks to the parameterization (\ref{Qbetap}), computing the minimum volume outer ellipsoid (MVOE) for the $p$-sum $\mathcal{E}(\bm{0},\bm{Q}_{1}) +_{p} \mathcal{E}(\bm{0},\bm{Q}_{2})$, is reduced to computing $\beta^{*}_{{\rm{vol}}}$ introduced in the previous Section. Motivated by the observation that the first order optimality criterion (\ref{DerivativeEqualsZero}) can be rearranged as
\begin{eqnarray}
\beta^{3-\frac{1}{p}}\displaystyle\sum_{i=1}^{d}\displaystyle\frac{\lambda_{i}}{1 + \beta^{\frac{1}{p}}\lambda_{i}} = \displaystyle\frac{1}{1 + \beta^{\frac{1}{p}}\lambda_{i}},
\label{MotivateFPE}	
\end{eqnarray}
we consider the fixed point recursion
\begin{eqnarray}
\beta_{n+1} = g(\beta_{n}) := \left(\displaystyle\frac{\displaystyle\sum_{i=1}^{d}\displaystyle\frac{1}{1 + \beta_{n}^{\frac{1}{p}}\lambda_{i}}}{\displaystyle\sum_{i=1}^{d}\displaystyle\frac{\lambda_{i}}{1 + \beta_{n}^{\frac{1}{p}}\lambda_{i}}}\right)^{\!\!\frac{p}{3p-1}},
\label{FPE}	
\end{eqnarray}
where $n=0,1,2,\hdots$, and $p \in[1,\infty)\setminus\{2\}$. Furthermore, $g : \mathbb{R}_{+} \mapsto \mathbb{R}_{+}$, i.e., the map $g$ is cone preserving. In the following Theorem, we show that the fixed point recursion (\ref{FPE}) converges to a unique positive root (Fig. \ref{FixedPtFn}), and is in fact contractive in the Hilbert metric \cite[Ch. 2]{LemmensNussbaumBook2012}. Thus, the recursion (\ref{FPE}) is indeed an efficient numerical algorithm to compute $\beta^{*}_{{\rm{vol}}}$. The recursion (\ref{FPE}) and the contraction proof below subsume our previous result \cite[Section IV.B]{HalderMinkSum2018} for the $p=1$ case (computing MVOE for the Minkowski sum).

\begin{theorem}\label{FPEthm}
Starting from any initial guess $\beta_{0} \in \mathbb{R}_{+}$, the recursion (\ref{FPE}) with fixed $p \in[1,\infty)\setminus\{2\}$, converges to a unique fixed point $\beta^{*}_{{\rm{vol}}}\in\mathbb{R}_{+}$, i.e., $\displaystyle\lim_{n\rightarrow\infty}g^{n}(\beta_{0})=\beta^{*}_{{\rm{vol}}}$.
\end{theorem}
\begin{proof}
For $\lambda_{i},x>0$, consider the positive functions $f_{i}:=1/\left(1+\lambda_{i} x^{1/p}\right)$, where $i=1,\hdots,d$, $p \in[1,\infty)\setminus\{2\}$, and let
\[\phi(x) := x^{\frac{p}{3p-1}}, \quad \text{and} \quad \psi(x) := \displaystyle\frac{\sum_{i}f_{i}}{\sum_{i}\lambda_{i}f_{i}}.\]
Clearly, $\phi(x)$ and $\psi(x)$ are both concave and increasing in $\mathbb{R}_{+}$, and therefore \cite[p. 84]{BoydCvxBook} so is $g(\beta_{n}) = \phi(\psi(\beta_{n}))$ as a function of $\beta_{n}$, $n=0,1,2,\hdots$. Consequently (see e.g., the first step in the proof of Theorem 2.1.11 in \cite{KrauseBook2015}) the map $g$ is contractive in Hilbert metric on the cone $\mathbb{R}_{+}$. By Banach contraction mapping theorem, $g$ admits	unique fixed point $\beta^{*}_{{\rm{vol}}}\in\mathbb{R}_{+}$, and $\displaystyle\lim_{n\rightarrow\infty}g^{n}(\beta_{0})=\beta^{*}_{{\rm{vol}}}$.
\end{proof}
The rate-of-convergence for (\ref{FPE}) is fast in practice, see Fig. \ref{FPConv}. Next, we show how the $p$-sum computation may arise in the reachability analysis for linear systems.

\begin{figure}[t]
 \centering
\includegraphics[width=\linewidth]{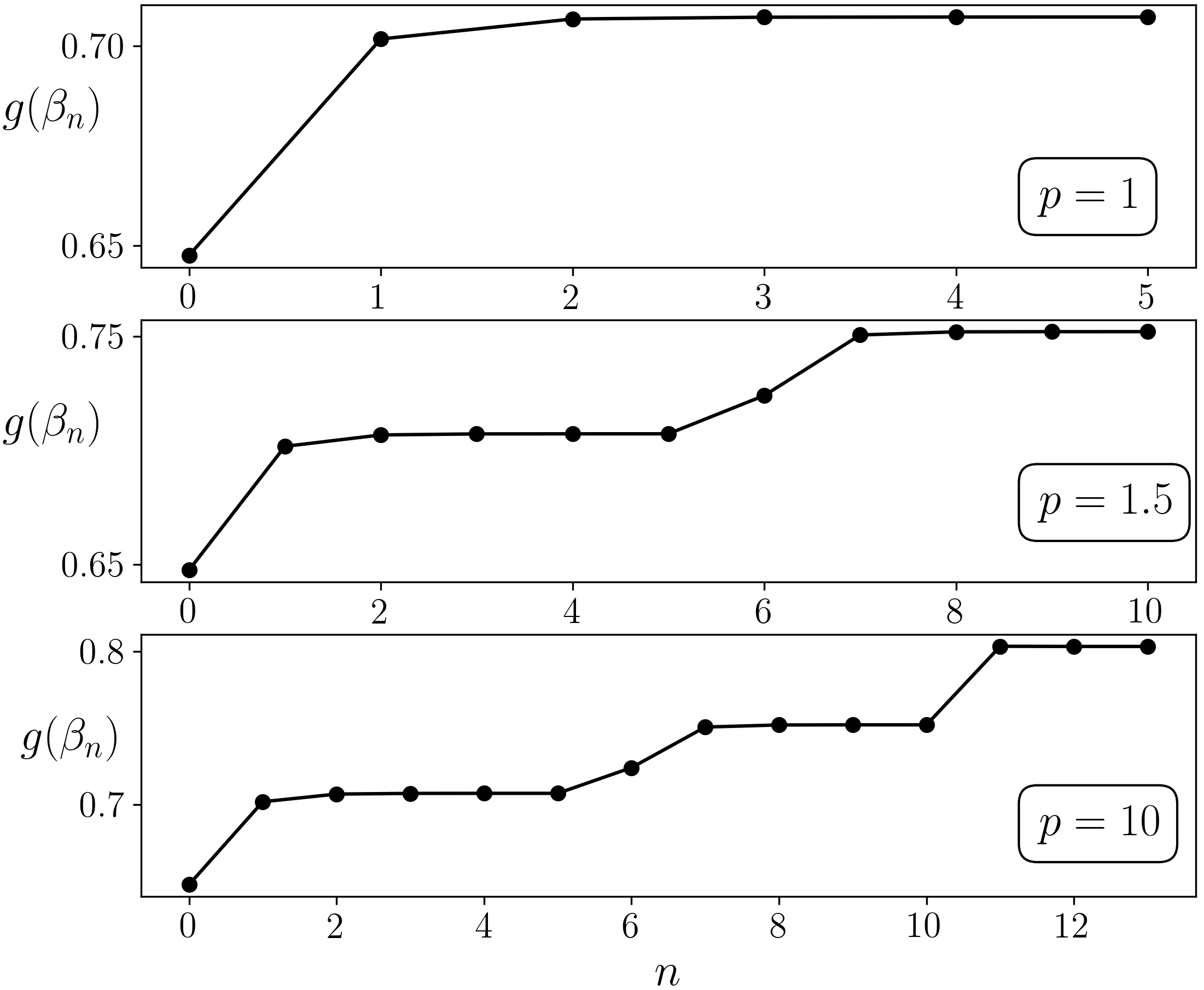}
 \caption{Starting from the same initial guess $\beta_{0} = 0.647584$, the iterates for (\ref{FPE}) are shown to converge (within an error tolerance of $10^{-5}$) in few steps for $p=1,1.5,10$. The above results are for the same parameters as in Fig. \ref{FixedPtFn}.}
 \label{FPConv}
\end{figure}

\subsection{Application to Reachability Analysis for Linear Systems}
Consider a discrete-time linear time-invariant (LTI) system 
\begin{eqnarray}
\bm{x}(t+1) = \bm{F}\bm{x}(t) + \bm{G}\bm{u}(t),	 \quad \bm{x}\in\mathbb{R}^{n_{x}}, \:\bm{u}\in\mathbb{R}^{n_{u}},
\label{LinearSys}
\end{eqnarray}
with set-valued uncertainties in initial condition $\bm{x}(0)\in\mathcal{X}_{0}\subset\mathbb{R}^{n_{x}}$, and control $\bm{u}(t)\in\mathcal{U}(t)\subset\mathbb{R}^{n_{u}}$. For $t=0,1,2,\hdots$, we would like to compute a tight outer ellipsoidal approximation of the reach set $\mathcal{X}(t)\ni \bm{x}(t)$, when the sets $\mathcal{X}_{0}$ and $\mathcal{U}(t)$ are modeled as $p_{1}$ and $p_{2}$-sum of ellipsoids, respectively, i.e.,
\begin{subequations}\label{pSumInitCondnControl}
\begin{align}
\label{pSumInitCondn}
	\mathcal{X}_{0}\! &= \mathcal{E}\left(\bm{x}_{0},\bm{Q}_{01}\right)\! +_{p_{1}} \hdots +_{p_{1}}\!\mathcal{E}\left(\bm{x}_{0},\bm{Q}_{0m}\right),\\
\label{pSumInitControl}	
	\hspace*{-0.07in}\mathcal{U}(t)\! &= \mathcal{E}\left(\bm{u}_{c}(t),\bm{U}_{1}(t)\right)\! +_{p_{2}} \hdots +_{p_{2}}\!\mathcal{E}\left(\bm{u}_{c}(t),\bm{U}_{n}(t)\right),
	\end{align}
\end{subequations} 
where $p_{1},p_{2}\geq 1$. Here, $\bm{x}_{0}$ and $\bm{u}_{c}(t)$ are the nominal initial condition and control, respectively. Further, $\bm{Q}_{0i}\in\mathbb{S}_{+}^{n_{x}}$ for $i=1,\hdots,m$, and for all $t=0,1,2,\hdots$, the matrices $\bm{U}_{j}(t)\in\mathbb{S}_{+}^{n_{u}}$ for $j=1,\hdots,n$. The case $m=n=1$ corresponds to the situation when both the initial condition and control have ellipsoidal uncertainties, and has appeared extensively in systems-control literature \cite{schweppe1968recursive, bertsekas1971recursive, schweppe1973, Chernousko1980PartI, MaksarovNorton1996, KurzhanskiValyi1997,durieu2001multi,chernousko2004properties, KurzhanskiyVaraiyaCDC2006, kurzhanskiy2007ellipsoidal,gagarinov2012computation}. By varying $p_{1},p_{2}\in[1,\infty]$ in (\ref{pSumInitCondnControl}), one obtains a large class of convex set-valued uncertainty descriptions.

In the absence of control ($\bm{u}(t)\equiv \bm{0}$), the set $\mathcal{X}(t)$ remains a $p_{1}$-sum of ellipsoids whenever  $\mathcal{X}_{0}$ is a $p_{1}$-sum of ellipsoids as in (\ref{pSumInitCondn}). This follows from the support function $h_{\mathcal{X}(t)}(\bm{y}) = h_{\bm{\Phi}(t,0)\mathcal{X}_{0}}(\bm{y}) = h_{\mathcal{X}_{0}}\left((\bm{\Phi}(t,0))^{\top}\bm{y}\right)$, $\bm{y}\in\mathbb{R}^{n_{x}}$, where $\bm{\Phi}(t,0) = \bm{F}^{t}$ is the state transition matrix of (\ref{LinearSys}). 

When the control $\bm{u}(t)$ is not identically zero, then the solution $\bm{x}(t) = \bm{F}^{t} \bm{x}(0) + \sum_{k=0}^{t-1}\bm{F}^{t-k-1}\bm{G}\bm{u}(k)$ corresponds to the support function
\begin{eqnarray}
h_{\mathcal{X}(t)}(\bm{y}) = h_{\mathcal{X}_{0}}\left((\bm{F}^{\top})^{t}\bm{y}\right) + \displaystyle\sum_{k=0}^{t-1} h_{\mathcal{U}(k)}\left(\bm{G}^{\top}(\bm{F}^{\top})^{t-k-1}\bm{y}\right).
\label{dLTVSptFn}	
\end{eqnarray}
From (\ref{DistributiveLinTrans}) and (\ref{dLTVSptFn}), it is evident that $\mathcal{X}(t)$ is the 1-sum (Minkowski sum) of $(t+1)$ convex sets, one of which is $p_{1}$-sum of ellipsoids, and each of the remaining $t$ sets are $p_{2}$-sum of ellipsoids, where $p_{1},p_{2}\geq 1$. Clearly, this remains true even when the system matrices $\bm{F}, \bm{G}$ are time-varying. We mention here that the general idea of using support functions for reachability analysis has appeared before in the literature \cite{witsenhausen1968sets, girard2008efficient}. 


In the following Section, we apply the formula (\ref{betamintr}) and the fixed point algorithm (\ref{FPE}) to compute the minimum trace and minimum volume outer ellipsoidal approximations of the reach sets of (\ref{LinearSys}), constrained in the parametric family (\ref{Qbetap}). With slight abuse of nomenclature, we will hereafter refer these ellipsoids as MTOE and MVOE, respectively. These should be understood as the optimal \emph{within the parametric family (\ref{Qbetap})}. The prototype algorithms for computing the same for the convex set $\mathcal{E}\left(\bm{0},\bm{Q}_{1}\right) +_{p} \mathcal{E}\left(\bm{0},\bm{Q}_{2}\right)$, $\bm{Q}_{1},\bm{Q}_{2}\in\mathbb{S}^{d}_{+}$, are given below. These serve as the building blocks for computing the parametric MTOE and MVOE for the reach set of (\ref{LinearSys}) subject to (\ref{pSumInitCondnControl}), which we illustrate next. For the results reported in this paper which use Algorithm \ref{AlgoMVOE}, we have set tol $= 10^{-5}$, and MaxIter $= 100$.

\begin{algorithm}
    \caption{Algorithm to compute the parametric MTOE for the $p$-sum $\mathcal{E}\left(\bm{0},\bm{Q}_{1}\right) +_{p} \mathcal{E}\left(\bm{0},\bm{Q}_{2}\right)$, $p\in[1,\infty)\setminus\{2\}$.}
    \label{AlgoMTOE}
    \begin{algorithmic}[1] 
        \Procedure{MTOE}{$\bm{Q}_{1},\bm{Q}_{2},p$} \Comment{$\bm{Q}_{1},\bm{Q}_{2}\in\mathbb{S}^{d}_{+}$}
            \State $\beta^{*}_{\text{tr}}\gets \left(\tr\left(\bm{Q}_{1}\right)/\tr\left(\bm{Q}_{2}\right)\right)^{\frac{p}{1+p}}$
            \State $\bm{Q}\gets \left(1 + 1/\beta^{*}_{\text{tr}}\right)^{\frac{1}{p}}\bm{Q}_{1} + \left(1 + \beta^{*}_{\text{tr}}\right)^{\frac{1}{p}}\bm{Q}_{2}$ 
            \State \textbf{return} $\mathcal{E}\left(\bm{0},\bm{Q}\right)$ \Comment{The parametric MTOE}
        \EndProcedure
    \end{algorithmic}
\end{algorithm}

\begin{algorithm}
    \caption{Algorithm to compute the parametric MVOE for the $p$-sum $\mathcal{E}\left(\bm{0},\bm{Q}_{1}\right) +_{p} \mathcal{E}\left(\bm{0},\bm{Q}_{2}\right)$, $p\in[1,\infty)\setminus\{2\}$.}
    \label{AlgoMVOE}
    \begin{algorithmic}[1] 
        \Procedure{MVOE}{$\bm{Q}_{1},\bm{Q}_{2},p,$ tol, MaxIter} \Comment{$\bm{Q}_{1},\bm{Q}_{2}\in\mathbb{S}^{d}_{+}$, tol is numerical tolerance, MaxIter is maximum number of iterations}
        	\State $\{\lambda_{i}\}_{i=1}^{d}\gets$ spectrum $\left(\bm{Q}_{1}^{-1}\bm{Q}_{2}\right)$
        	\State $\beta\gets$ random positive number \Comment{Initialize parameter}
        	\State $\varepsilon \gets 1$ \Comment{Initialize error}
        	\State $j \gets 1$ \Comment{Initialize iteration index}
        	\While{(($\varepsilon >$ tol) $\&$ ($j <$ MaxIter))}
        	    \State $\beta_{\text{new}} \gets g(\beta)$ \Comment{$g(\cdot)$ as in (\ref{FPE}), needs $\{\lambda_{i}\}_{i=1}^{d},p$}
                \State $\varepsilon \gets |\beta_{\text{new}} - \beta|$
                \State $\beta \gets \beta_{\text{new}}$
            \EndWhile
            \State $\beta^{*}_{\text{vol}}\gets \beta$
            \State $\bm{Q}\gets \left(1 + 1/\beta^{*}_{\text{vol}}\right)^{\frac{1}{p}}\bm{Q}_{1} + \left(1 + \beta^{*}_{\text{vol}}\right)^{\frac{1}{p}}\bm{Q}_{2}$ 
            \State \textbf{return} $\mathcal{E}\left(\bm{0},\bm{Q}\right)$ \Comment{The parametric MVOE}
        \EndProcedure
    \end{algorithmic}
\end{algorithm}


\begin{figure*}[t]
 \centering
\includegraphics[width=0.92\linewidth]{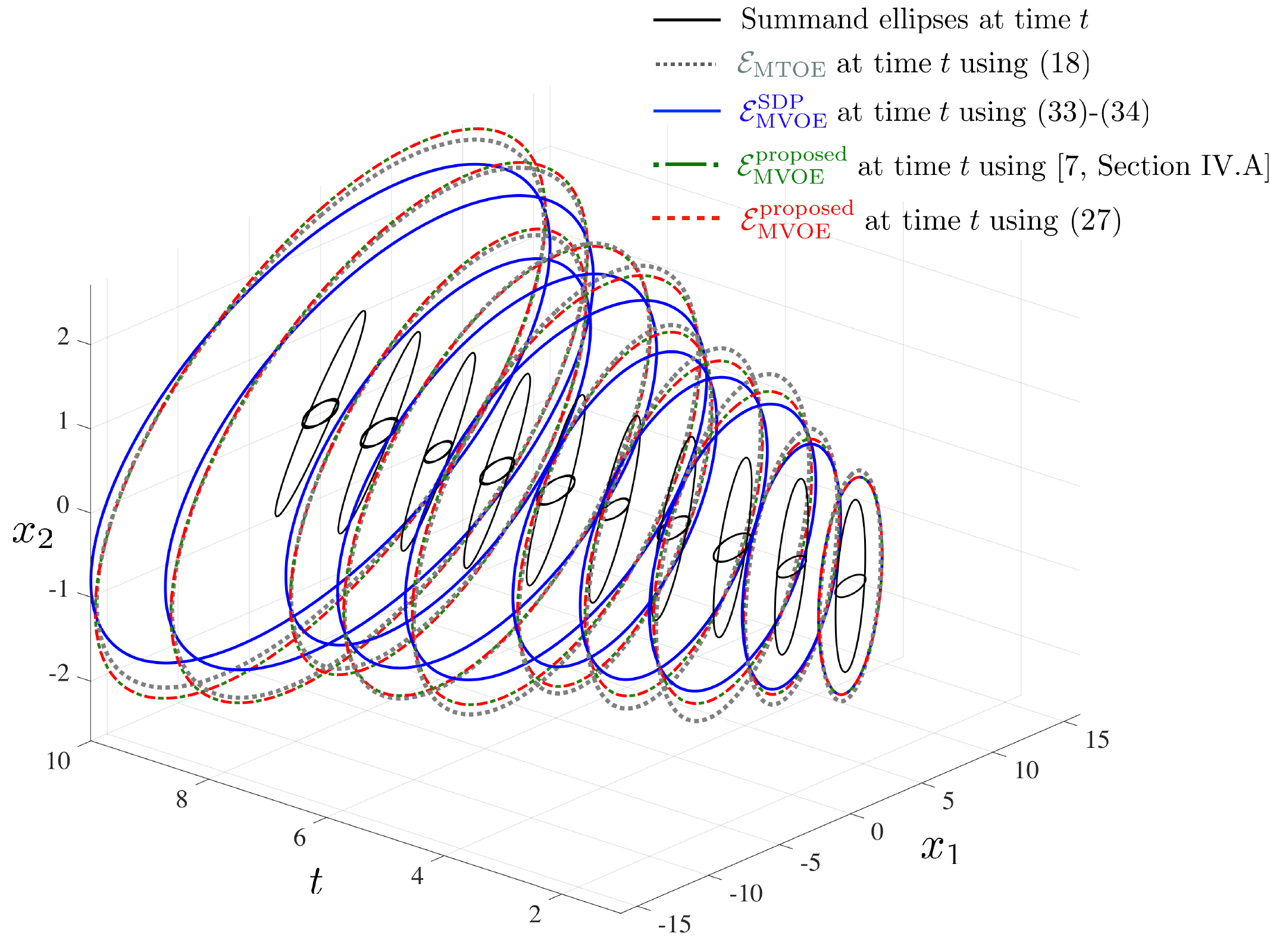}
 \caption{For dynamics (\ref{ExampleLTIdiscretetime}), the MTOEs and MVOEs of the reach sets $\mathcal{X}(t)$ given by (\ref{EllipsoidalMinkSum}) with parameters (\ref{ParamsMinkSum}), are shown for $t=1,\hdots,10$, along with the $(t+1)$ summand ellipses in the Minkowski sum (\ref{EllipsoidalMinkSum}) for each $t$. The MTOEs admit analytical solution (\ref{betamintr}). We compute the MVOEs via three different methods: SDP computation (\ref{BoydSDP})-(\ref{Constr}), a root bracketing technique to solve (\ref{AfterDivisionBylambdaprod}) proposed in \cite[Section IV.A]{HalderMinkSum2018}, and recursion (\ref{FPE}).}
 \label{ReachSetPlot}
\end{figure*}

\section{Numerical Simulations}
We consider the planar discrete-time linear system 
{\small{\begin{eqnarray}\label{ExampleLTIdiscretetime}
\begin{pmatrix}
 	x_{1}(t+1)\\
 	x_{2}(t+1)
 \end{pmatrix}
 = \begin{pmatrix}
 1 & h\\
 0 & 1	
 \end{pmatrix}
\begin{pmatrix}x_{1}(t)\\
 x_{2}(t)	
 \end{pmatrix} \: + \: \begin{pmatrix}
 h & h^{2}/2\\
 0 & h	
 \end{pmatrix}
\begin{pmatrix}u_{1}(t)\\
 u_{2}(t)	
 \end{pmatrix},	
\end{eqnarray}}}
which can be seen as the sampled version of the continuous-time system $\dot{x}_{1}=x_{2}+u_{1},\,\dot{x}_{2}=u_{2}$, with sampling period $h>0$. To illustrate the proposed algorithms, we will compute the reach sets of (\ref{ExampleLTIdiscretetime}) for various convex uncertainty models of the form (\ref{pSumInitCondnControl}) for $\mathcal{X}_{0}$ and $\mathcal{U}(t)$. For this numerical example, we set the nominal initial condition $\bm{x}_{0} \equiv \bm{0}$, and nominal control $\bm{u}_{c}(t)\equiv\bm{0}$. For other examples, see Appendix \ref{AppendixDoubleIntegrator} and \ref{AppendixISS}.

\subsection{The case $m=n=1$}\label{ReachabilitySubsection}
For $m=n=1$ in (\ref{pSumInitCondnControl}), both $\mathcal{X}_{0}$ and $\mathcal{U}(t)$ are ellipsoidal, and hence the reach set $\mathcal{X}(t)$ for (\ref{ExampleLTIdiscretetime}), is the Minkowski sum of $(t+1)$ ellipsoids. Consequently, we are led to compute the MTOE and MVOE of the Minkowski sum\footnote[3]{The symbol ``$\Sigma$" in (\ref{EllipsoidalMinkSum}) stands for the 1-sum.}
\begin{align}
	\mathcal{X}(t) = \bm{F}^{t}\mathcal{E}\left(\bm{0},\bm{Q}_{0}\right) \:+_{1}\: \displaystyle\sum_{k=0}^{t-1} \bm{F}^{t-k-1}\bm{G}\mathcal{E}\left(\bm{0},\bm{U}(t)\right).
	\label{EllipsoidalMinkSum}
\end{align}
Notice that the MTOE admits analytical solution (\ref{betamintr}), applied pairwise to the $(t+1)$ summand ellipsoids in (\ref{EllipsoidalMinkSum}), with $p=1$. For this case, the current state-of-the-art for MVOE computation is to reformulate the same as a semidefinite programming (SDP) problem via the $\mathcal{S}$-procedure (see e.g., \cite[Ch. 3.7.4]{BoydLMIBook}). Specifically, given $(t+1)$ ellipsoids $\mathcal{E}(\bm{q}_{i},\bm{Q}_{i})$ or equivalently $\mathcal{E}(\bm{A}_{i},\bm{b}_{i},c_{i})$ in $\mathbb{R}^{n_{x}}$, $i=1,\hdots,t+1$, to compute the MVOE containing their Minkowski sum $\mathcal{E}(\bm{q}_{1},\bm{Q}_{1}) +_{1} \hdots +_{1} \mathcal{E}(\bm{q}_{t+1},\bm{Q}_{t+1})$, one solves the SDP problem:
\begin{eqnarray}
\underset{\bm{A}_{0},\bm{b}_{0},\tau_{1},\hdots,\tau_{t+1}}{\text{minimize}}\;\log\det\,\bm{A}_{0}^{-1}	
\label{BoydSDP}
\end{eqnarray}
subject to
	\begin{subequations}
	\begin{alignat}{2}
		&\bm{A}_{0} \succ \bm{0}, \label{posidef}\\
		&\tau_{i} \geq 0, \qquad i=1,\hdots,t+1, \label{nonneg}\\
		&\begin{bmatrix}
		\bm{E}_{0}^{\top}\bm{A}_{0}\bm{E}_{0} & \bm{E}_{0}^{\top}\bm{b}_{0} & \bm{0}\\
		\bm{b}_{0}^{\top}\bm{E}_{0} & -1 & \bm{b}_{0}^{\top}\\
		\bm{0} & \bm{b}_{0} & -\bm{A}_{0}	
		\end{bmatrix} \!-\! \displaystyle\sum_{i=1}^{t+1}\tau_{i}\begin{bmatrix}
	\widetilde{\bm{A}}_{i} & \widetilde{\bm{b}}_{i} & \bm{0}\\
	\widetilde{\bm{b}}_{i}^{\top} & c_{i} & \bm{0}\\
	\bm{0} & \bm{0} & \bm{0}
\end{bmatrix}
 \preceq \bm{0}, \label{LMI}
	\end{alignat}	
	\label{Constr}
\end{subequations}
where $\bm{E}_{i}$ is the binary matrix of size $n_{x}\times (t+1) n_{x}$ that selects the $i$-th vector, $i=1,\hdots,t+1$, from the vertical stacking of $(t+1)$ vectors, each of size $n_{x}\times 1$; and
\[\bm{E}_{0} := \displaystyle\sum_{i=1}^{t+1}\bm{E}_{i}, \, \widetilde{\bm{A}}_{i} := \bm{E}_{i}^{\top}\bm{A}_{i}\bm{E}_{i}, \, \widetilde{\bm{b}}_{i}:=\bm{E}_{i}^{\top}\bm{b}_{i}, \, i=1,\hdots,t+1.\]
The argmin pair $(\bm{A}_{0}^{*},\bm{b}_{0}^{*})$ associated with the SDP (\ref{BoydSDP})-(\ref{Constr}), results the optimal ellipsoid
\[\mathcal{E}_{\text{MVOE}}^{\text{SDP}} := \mathcal{E}\left(\bm{q}_{\text{SDP}},\bm{Q}_{\text{SDP}}\right),\]
where, using (\ref{Ab2Qq}), $\bm{Q}_{\text{SDP}}:=(\bm{A}_{0}^{*})^{-1}$, and $\bm{q}_{\text{SDP}}:=-\bm{Q}_{\text{SDP}}\bm{b}_{0}^{*}$. 

Our intent is to compare $\mathcal{E}_{\text{MVOE}}^{\text{SDP}}$ with $\mathcal{E}_{\text{MVOE}}^{\text{proposed}}$, where
\[\mathcal{E}_{\text{MVOE}}^{\text{proposed}} := \mathcal{E}\left(\bm{q}_{1} + \hdots + \bm{q}_{t+1}, \bm{Q}(\beta_{\text{vol}}^{*})\right),\]
and the parametric form of $\bm{Q}(\beta_{\text{vol}}^{*})$ is given by (\ref{Qbetap}) with $p=1$, applied pairwise to the given set of shape matrices $\{\bm{Q}_{1}, \hdots, \bm{Q}_{t+1}\}$. The numerical value of $\beta_{\text{vol}}^{*}$ is computed from the fixed point recursion (\ref{FPE}) with $p=1$, solved pairwise from the set $\{\bm{Q}_{1}, \hdots, \bm{Q}_{t+1}\}$. In other words, due to the associate property of $p$-sum, the ellipsoid $\mathcal{E}^{\text{proposed}}_{\text{MVOE}}$ is obtained by applying Algorithm \ref{AlgoMVOE} pairwise to the set $\{\bm{Q}_{1}, \hdots, \bm{Q}_{t+1}\}$, and then translating the resulting ellipsoid by vector $\bm{q}_{1} + \hdots + \bm{q}_{t+1}$.

We will see that the MVOE algorithms proposed herein help in reducing computational time, compared to the SDP approach, without sacrificing accuracy. For comparing numerical performance, we implemented both the SDP (via {\texttt{cvx}} \cite{cvx}) and our proposed algorithms in MATLAB 2016b, on 2.6 GHz Intel Core i5 processor with 8 GB memory.

For the dynamics (\ref{ExampleLTIdiscretetime}), we set $h = 0.3$, and
\begin{eqnarray}
\bm{Q}_{0} = \bm{I}_{2}, \quad	\bm{U}(t) = \left(1 + \cos^{2}(t)\right){\rm{diag}}([10,0.1]),
\label{ParamsMinkSum}
\end{eqnarray}
and for each $t=1,2,\hdots$, compute the MTOE and MVOE of the reach set (\ref{EllipsoidalMinkSum}). For MVOE computation, we use three different methods: using the SDP (\ref{BoydSDP})-(\ref{Constr}), using a root-bracketing algorithm proposed in \cite[Section IV.A]{HalderMinkSum2018} to solve (\ref{AfterDivisionBylambdaprod}), and by using the fixed point recursion (\ref{FPE}) proposed herein. In Fig. \ref{ReachSetPlot}, the corresponding MTOEs and MVOEs, as well as the summand ellipses in the Minkowski sum (\ref{EllipsoidalMinkSum}) are shown for $t=1,\hdots,10$. In this paper, we do not emphasize the root bracketing method for MVOE computation given in \cite[Section IV.A]{HalderMinkSum2018} since that is a custom method for $n_{x} = 2$, while the SDP (\ref{BoydSDP})-(\ref{Constr}) and the fixed point recursion (\ref{FPE}) are valid in any dimensions.

\begin{center}
\begin{table*}[t]
\centering
	\begin{tabular}{|l||*{10}{c|}}\hline
\backslashbox{${\rm{vol}}$(MVOE)}{Physical time step}
&$t=1$&$t=2$&$t=3$
&$t=4$&$t=5$ &$t=6$&$t=7$&$t=8$&$t=9$&$t=10$\\\hline\hline
&&&&&&&&&&\\
${\rm{vol}}\left(\mathcal{E}_{\text{MVOE}}^{\text{SDP}}\right)$ using (\ref{BoydSDP})-(\ref{Constr}) & \makebox[4em]{8.6837} & \makebox[4em]{14.5461} & \makebox[4em]{27.9035} & \makebox[4em]{31.9097}  & \makebox[4em]{35.0421} & \makebox[4em]{61.0650} & \makebox[4em]{65.3182} & \makebox[4em]{59.1310} & \makebox[4em]{100.8786}  & \makebox[4em]{111.2311}\\
&&&&&&&&&&\\ \hline
&&&&&&&&&&\\ 
${\rm{vol}}\left(\mathcal{E}_{\text{MVOE}}^{\text{proposed}}\right)$ using \cite[Section IV.A]{HalderMinkSum2018} & \makebox[4em]{8.6837} & \makebox[4em]{14.6765} &  \makebox[4em]{28.7263} &  \makebox[4em]{33.2574} &  \makebox[4em]{36.8740} & \makebox[4em]{65.1379} & \makebox[4em]{70.1631} & \makebox[4em]{63.8502} & \makebox[4em]{109.2246} & \makebox[4em]{120.8542}\\
&&&&&&&&&&\\ \hline
&&&&&&&&&&\\
${\rm{vol}}\left(\mathcal{E}_{\text{MVOE}}^{\text{proposed}}\right)$ using (\ref{FPE}) & \makebox[4em]{8.6837} & \makebox[4em]{14.6765} & \makebox[4em]{28.7263} &  \makebox[4em]{33.2574} & \makebox[4em]{36.8740} & \makebox[4em]{65.1379} & \makebox[4em]{70.1632} & \makebox[4em]{63.8502} & \makebox[4em]{109.2246} & \makebox[4em]{120.8542}\\
&&&&&&&&&&\\\hline
\end{tabular}
\caption{Comparison of the volumes of the MVOEs for the reach set $\mathcal{X}(t)$ in (\ref{EllipsoidalMinkSum}) at $t=1,2,\hdots, 10$, corresponding to the dynamics (\ref{ExampleLTIdiscretetime}) with parameters given by (\ref{ParamsMinkSum}), computed via three methods: by solving the SDP (\ref{BoydSDP})-(\ref{Constr}) (\emph{first row}), by using a pairwise root bracketing technique reported in \cite[Section IV.A]{HalderMinkSum2018} (\emph{second row}), and by iterating the fixed point recursion (\ref{FPE}) with tolerance $10^{-5}$ (\emph{third row}). The corresponding MVOEs are depicted in Fig. \ref{ReachSetPlot}. Overall, the MVOE volumes in the second and third row are in close agreement, while they are slightly conservative than the SDP results in the first row.}
\label{TableVolComparison}
\end{table*}
\end{center}

To assess the quality of the outer ellipsoidal approximations shown in Fig. \ref{ReachSetPlot}, we compare the volumes (in our two-dimensional example, areas) of the MVOEs computed via the three different methods in Table \ref{TableVolComparison}. The columns in Table \ref{TableVolComparison} correspond to different time steps while the rows correspond to the three different methods mentioned above. From Table \ref{TableVolComparison}, notice that the MVOE volumes are not monotone in time for any given method (see e.g., the columns for $t=7$ and $t=8$) since the shape matrices $\bm{U}(t)$ in (\ref{ParamsMinkSum}) are periodic. We notice that the volumes listed in the second and third row in Table \ref{TableVolComparison} are in close agreement, while they are slightly conservative compared to the  same in the first row, which are computed by solving the SDP (\ref{BoydSDP})-(\ref{Constr}). Furthermore, the relative numerical error seems to grow (albeit slowly) with $t$ (with increasing number of summand ellipsoids).

By looking at the computational accuracy comparisons from Table \ref{TableVolComparison}, it may seem that the SDP approach is superior to the algorithms proposed herein. However, the corresponding computational runtimes plotted in Fig. \ref{FPCompTime} reveal that solving the fixed point recursion (\ref{FPE}) entails orders of magnitude speed-up compared to solving the SDP. Given that the growing interests in reach set computation among practitioners are stemming from real-time safety critical applications (e.g., decision making for collision avoidance in a traffic of autonomous and semi-autonomous cars or drones), the issue of computational runtime becomes significant. Such applications indeed require computing the reach set over a short physical time horizon (typically a moving horizon of few seconds length); the MVOE computational time-scale, then, needs to be much smaller than the dynamics time-scale. The results in Fig. \ref{FPCompTime} show that the proposed algorithms can be useful in such context as they offer significant computational speed-up without much conservatism.

\begin{figure}[t]
 \centering
\includegraphics[width=\linewidth]{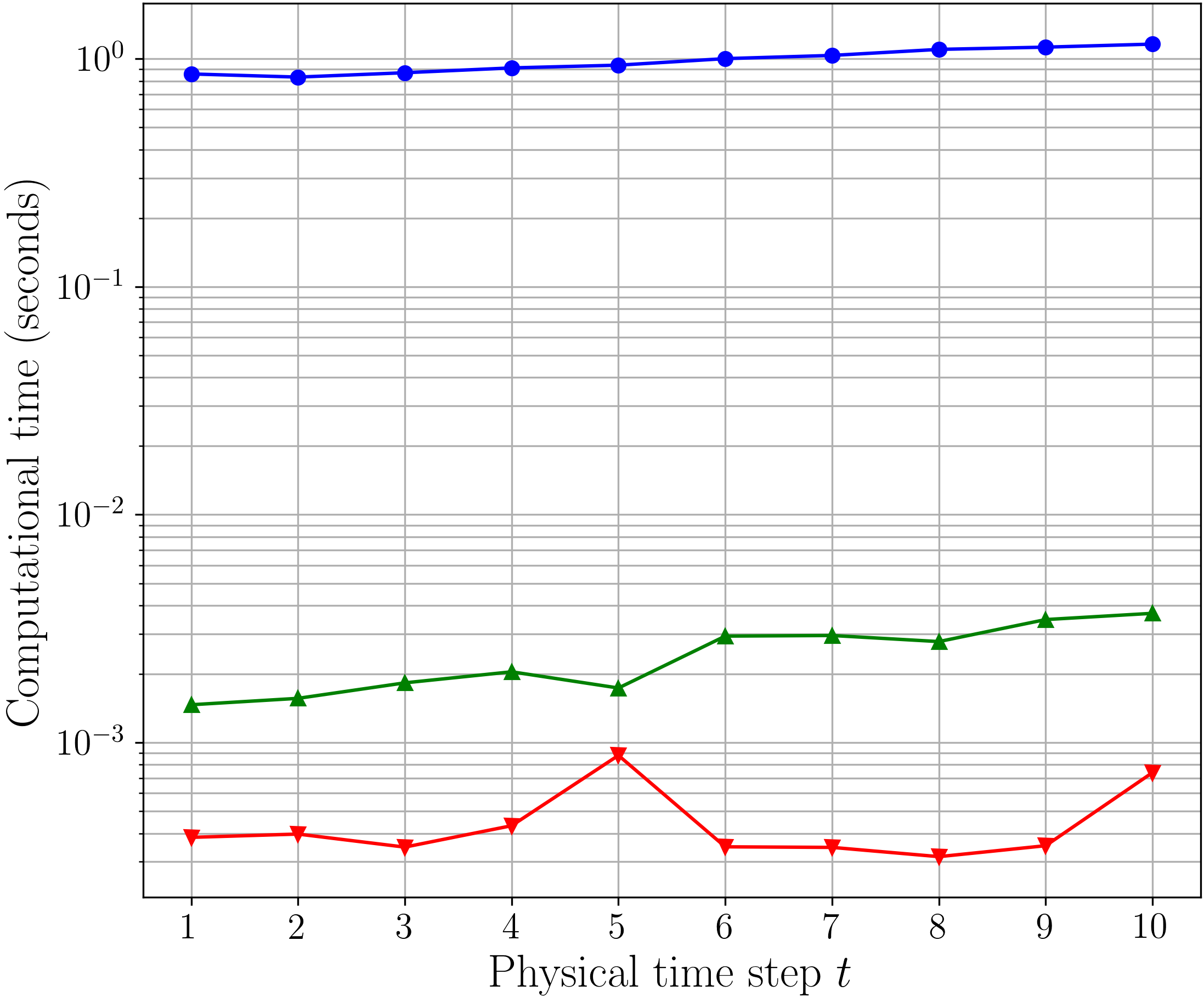}
 \caption{Comparison of the computational times for MVOE calculations in Fig. \ref{ReachSetPlot} and Table \ref{TableVolComparison}. The MVOE computational times for the proposed fixed point recursion method (\protect\invtriangleline, \emph{third row in Table \ref{TableVolComparison}}) are orders of magnitude faster than the same for the SDP computation (\protect\circleline, \emph{first row in Table \ref{TableVolComparison}}). A root bracketing technique to solve (\ref{AfterDivisionBylambdaprod}) proposed in \cite[Section IV.A]{HalderMinkSum2018} requires more computational time (\protect\triangleline, \emph{second row in Table \ref{TableVolComparison}}) than iterating the recursion (\ref{FPE}), but is faster than solving the SDP (\ref{BoydSDP})-(\ref{Constr}).}
 \label{FPCompTime}
\end{figure}

\subsection{The Case $m,n>1$}\label{subsecmnmorethan1}
\begin{figure*}[t]
 \centering
\includegraphics[width=0.92\linewidth]{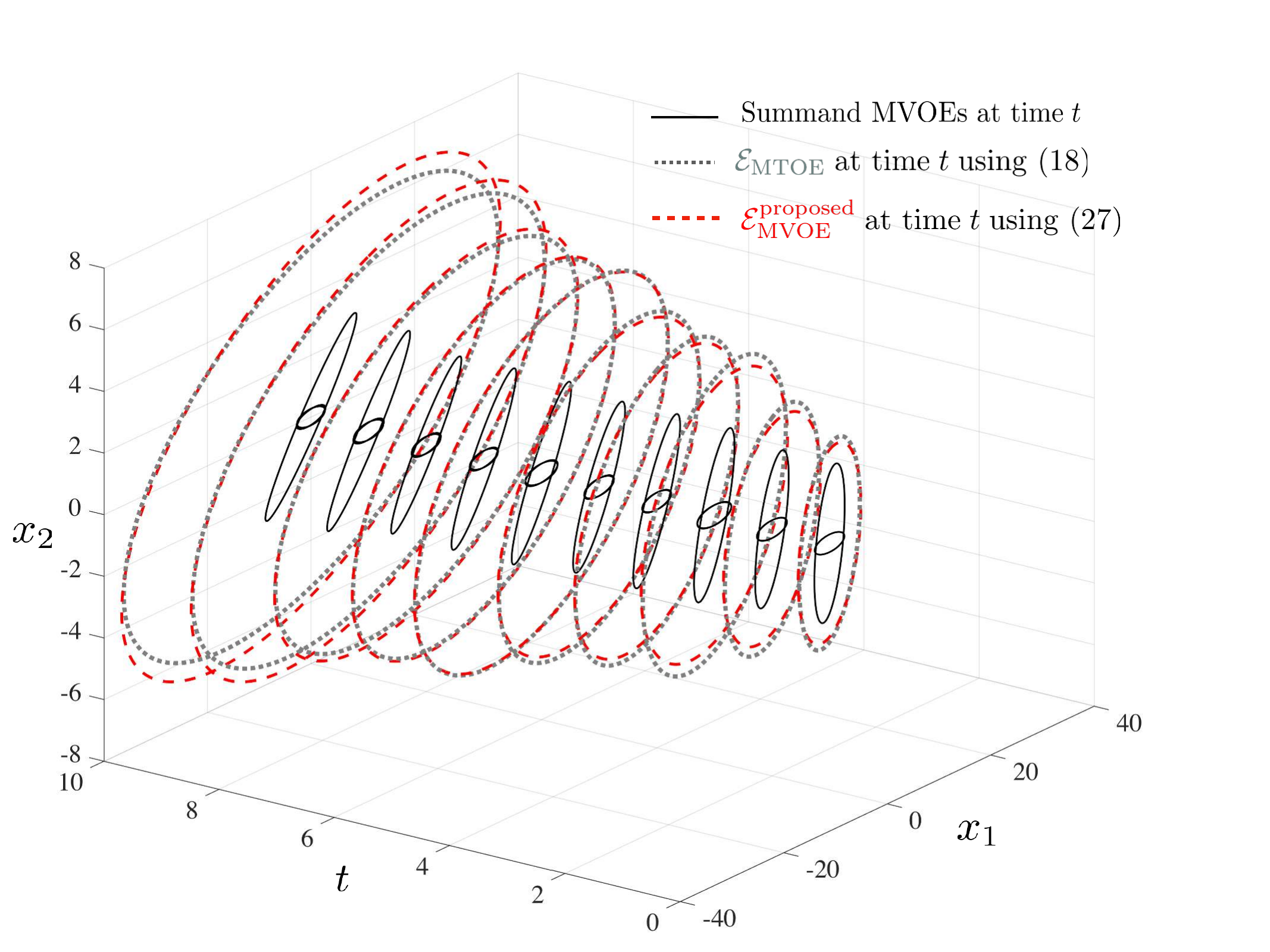}
 \caption{For dynamics (\ref{ExampleLTIdiscretetime}), the MTOEs and MVOEs of the reach sets $\mathcal{X}(t)$ given by (\ref{pEllipsoidalMinkSum}) with parameters (\ref{InitShapeMatrix}) and (\ref{ControlShapeMatrix}), are shown for $t=1,\hdots,10$, along with the $(t+1)$ summand MVOEs in the 1-sum (\ref{pEllipsoidalMinkSum}) for each $t$. As before, the MTOEs admit analytical solution (\ref{betamintr}), applied pairwise. We compute the MVOEs via the pairwise recursion (\ref{FPE}).}
 \label{pSumReachSetPlot}
\end{figure*}
We now consider the case $m,n>1$ in (\ref{pSumInitCondnControl}) with $m\neq n$, and compute the MVOEs and MTOEs for the reach set $\mathcal{X}(t)$ of (\ref{ExampleLTIdiscretetime}) with $h=0.3$, as before. Specifically, in (\ref{pSumInitCondnControl}), we fix $m=2, n=3$, $p_{1} = 2.5$, and $p_{2} = 1.5$. In words, the set of uncertain initial conditions $\mathcal{X}_{0}$ is modeled as $2.5$-sum of two ellipsoids, i.e., $\mathcal{X}_{0} = \mathcal{E}\left(\bm{0},\bm{Q}_{01}\right) +_{2.5} \mathcal{E}\left(\bm{0},\bm{Q}_{02}\right)$; the control uncertainty set $\mathcal{U}(t)$ is modeled as $1.5$-sum of three (time-varying) ellipsoids, i.e., $\mathcal{U}(t) = \mathcal{E}\left(\bm{0},\bm{U}_{1}(t)\right) +_{1.5} \mathcal{E}\left(\bm{0},\bm{U}_{1}(t)\right) +_{1.5} \mathcal{E}\left(\bm{0},\bm{U}_{3}(t)\right)$. The shape matrices for the set of uncertain initial conditions are randomly generated positive definite matrices:
\begin{eqnarray}
\bm{Q}_{01} = \begin{pmatrix}
 	    2.2259   & 0.1992\\
    0.1992    & 2.4357 \end{pmatrix}, \bm{Q}_{02} = \begin{pmatrix}
 	        2.3111  &  0.6768\\
    0.6768  &  2.1848 \end{pmatrix}.
\label{InitShapeMatrix}	
\end{eqnarray}
The shape matrices for the set of uncertain controls are chosen as
\begin{eqnarray}
\bm{U}_{j}(t) = \left(1 + \cos^{2}\left(jt\right)\right){\rm{diag}}([10,0.1]), \quad j=1,2,3.
\label{ControlShapeMatrix}	
\end{eqnarray}
Then, the reach set $\mathcal{X}(t)$ for (\ref{ExampleLTIdiscretetime}) at any time $t>0$, equals
{\small{\begin{align}
	&\bm{F}^{t}\bigg\{\!\mathcal{E}\left(\bm{0},\bm{Q}_{10}\right) \:+_{2.5}\: \mathcal{E}\left(\bm{0},\bm{Q}_{20}\right)\!\bigg\} +_{1} \displaystyle\sum_{k=0}^{t-1} \bm{F}^{t-k-1}\bm{G}\bigg\{\!\mathcal{E}\left(\bm{0},\bm{U}_{1}(t)\right) \nonumber\\
	&\:+_{1.5}\:\mathcal{E}\left(\bm{0},\bm{U}_{2}(t)\right)  \:+_{1.5}\: \mathcal{E}\left(\bm{0},\bm{U}_{3}(t)\right)\!\bigg\},
	\label{pEllipsoidalMinkSum}
\end{align}}}
which, due to (\ref{DistributiveLinTrans}), is the 1-sum\footnote[4]{The summation symbol ``$\Sigma$" in (\ref{pEllipsoidalMinkSum}) stands for the 1-sum.} of $(t+1)$ convex sets, one of them being the 2.5-sum of two ellipsoids, and the remaining $t$ of them each being the 1.5-sum of three ellipsoids.

\begin{center}
\begin{table*}[t]
\centering
	\begin{tabular}{|l||*{10}{c|}}\hline
\backslashbox{${\rm{vol}}$(MVOE)}{Physical time step}
&$t=1$&$t=2$&$t=3$
&$t=4$&$t=5$ &$t=6$&$t=7$&$t=8$&$t=9$&$t=10$\\\hline\hline
&&&&&&&&&&\\
${\rm{vol}}\left(\mathcal{E}_{\text{MVOE}}^{\text{proposed}}\right)$ using (\ref{FPE}) & \makebox[4em]{57.7493} & \makebox[4em]{99.3984} & \makebox[4em]{182.9045} &  \makebox[4em]{206.0490} & \makebox[4em]{266.6789} & \makebox[4em]{383.9408} & \makebox[4em]{387.4037} & \makebox[4em]{461.7879} & \makebox[4em]{610.9069} & \makebox[4em]{666.9160}\\
&&&&&&&&&&\\\hline
\end{tabular}
\caption{Volumes of the MVOEs for the reach set $\mathcal{X}(t)$ in (\ref{pEllipsoidalMinkSum}) at $t=1,2,\hdots, 10$, corresponding to the dynamics (\ref{ExampleLTIdiscretetime}) with parameters given by (\ref{InitShapeMatrix}) and (\ref{ControlShapeMatrix}), computed via fixed point recursion (\ref{FPE}) with tolerance $10^{-5}$. The corresponding MVOEs are shown in Fig. \ref{pSumReachSetPlot}.}
\label{TablepSumVolComparison}
\end{table*}
\end{center}

Unlike the case in Section \ref{ReachabilitySubsection} where the SDP approach is known in the literature for computing an MVOE of the 1-sum, and served as a baseline algorithm to compare the performance of our proposed algorithms, to the best of our knowledge, no such algorithm is known for the general $p$-sum case. Our proposed algorithms are generic enough to enable the MVOE/MTOE computation in this case. Specifically, at each time $t$, we first compute the MVOEs (resp. MTOEs) for each of the $(t+1)$ summand convex sets in (\ref{pEllipsoidalMinkSum}) by solving (\ref{FPE}) (resp. (\ref{betamintr})) pairwise, and then compute the MVOE (resp. MTOE) of the 1-sum of the resulting MVOEs (resp. resulting MTOEs) using the same. In Fig. \ref{pSumReachSetPlot}, we show the MVOEs and MTOEs thus computed, for the reach set (\ref{pEllipsoidalMinkSum}) at $t=1,\hdots,10$.

\begin{figure}[t]
 \centering
\includegraphics[width=\linewidth]{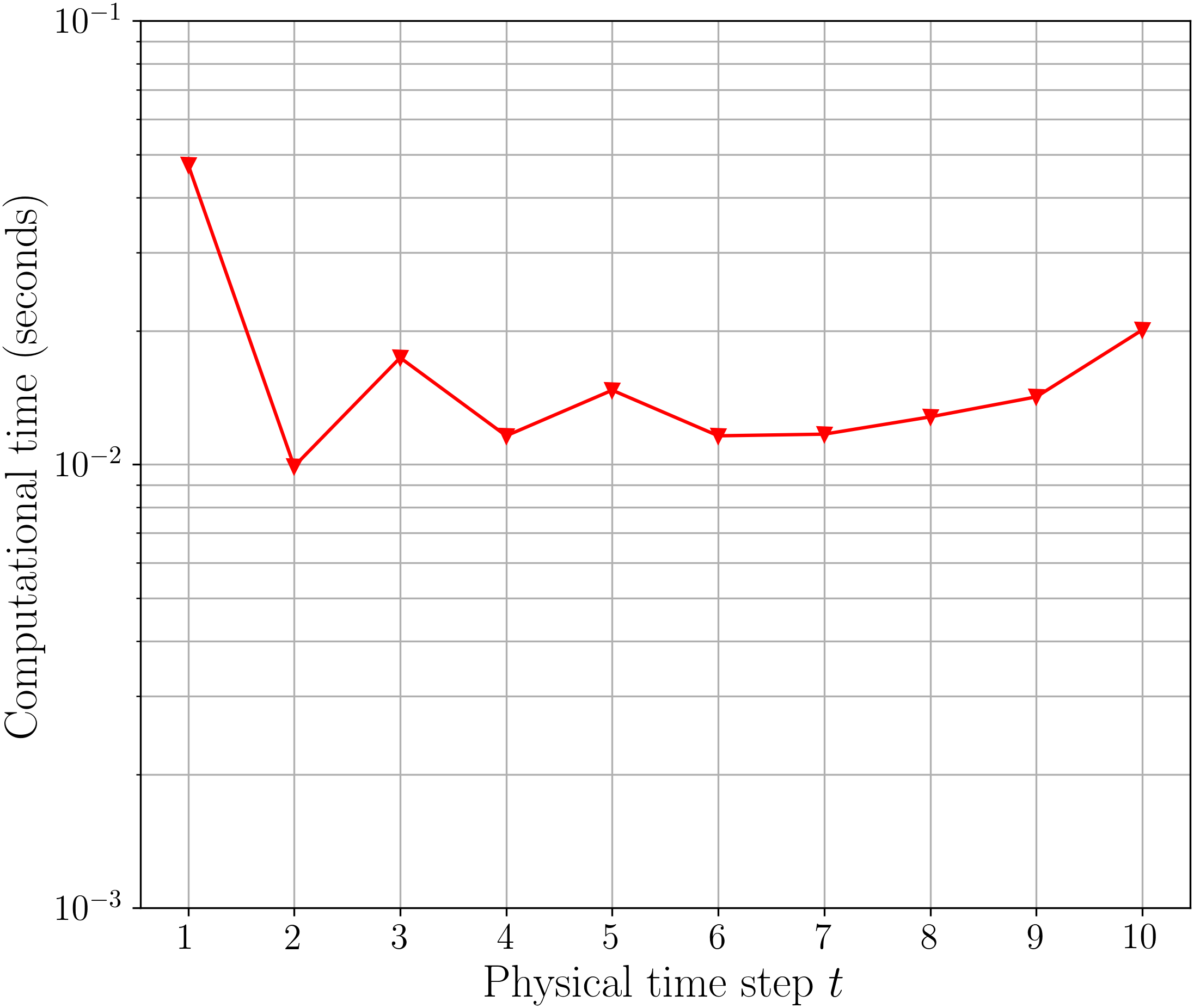}
 \caption{Computational times for the MVOE calculations in Fig. \ref{pSumReachSetPlot} and Table \ref{TablepSumVolComparison} using recursion (\ref{FPE}).}
 \label{FPpSumCompTime}
\end{figure}

The volumes (in our two dimensional numerical example, areas) of the MVOEs shown in Fig. \ref{pSumReachSetPlot}, are listed in Table \ref{TablepSumVolComparison}. The corresponding computational times are shown in Fig. \ref{FPpSumCompTime}. Notice that the MVOE computational times reported in Fig. \ref{FPpSumCompTime} are about two orders of magnitude slower than the same reported in Fig. \ref{FPCompTime}. This is expected since the results in Fig. \ref{FPpSumCompTime} correspond to computing MVOEs of the 1-sums while the same in Fig. \ref{FPpSumCompTime} correspond to computing MVOEs of the mixed $p$-sums (in our example, 1-sums, 1.5 sums and 2.5 sums) at any given time $t>0$, and is indeed consistent with the observation made in Fig. \ref{FPConv} that the rate-of-convergence of recursion (\ref{FPE}) decreases with increasing $p\geq 1$. Nevertheless, the computational times shown in Fig. \ref{FPpSumCompTime} are still smaller than the computational times for the SDP approach in the 1-sum case shown in Fig. \ref{FPCompTime}.

\subsection{Quality of Approximation}\label{subsecQualityOfApprox}
It is natural to investigate the quality of approximations for the MTOEs and MVOEs reported herein with respect to the actual reach sets, in terms of the ``shapes" of the true and approximating sets. For  example, it would be undesirable if the MTOE/MVOE computation promotes ``skinny ellipsoids" which are too elongated along some directions and too compressed along others, when the same may not hold for the actual reach sets. This motivates us to compute the (two-sided) Hausdorff distance $\delta_{{\rm{H}}}(t)$ (which is a metric) at each time step $t$, between the true reach set $\mathcal{X}(t)$ and its approximating outer ellipsoid $\widehat{\mathcal{E}}(t)$, given by 
\begin{eqnarray}
\delta_{{\rm{H}}}(t) := \max\bigg\{\sup_{\bm{x}(t)\in\mathcal{X}(t)}\inf_{\widehat{\bm{x}}(t)\in\widehat{\mathcal{E}}(t)} \parallel \bm{x}(t) - \widehat{\bm{x}}(t) \parallel_{2}\:, \nonumber\\
\sup_{\widehat{\bm{x}}(t)\in\widehat{\mathcal{E}}(t)}\inf_{\bm{x}(t)\in\mathcal{X}(t)} \parallel \bm{x}(t) - \widehat{\bm{x}}(t) \parallel_{2}\bigg\}.
\label{HausdorffDefn}	
\end{eqnarray}

\begin{remark}
We clarify here that the MVOE for any compact convex set (in our case, the reach set) is guaranteed to be unique, and is referred to as the L\"{o}wner-John ellipsoid \cite{john2014extremum,busemann1950foundations}.	 Exact computation of the L\"{o}wner-John ellipsoid, however, leads to semi-infinite programming \cite[Ch. 8.4.1]{BoydCvxBook}, and for most convex sets such as the Minkowski sum of ellipsoids, has no known exact SDP representation. Thus, the ellipsoid $\mathcal{E}_{\text{MVOE}}^{\text{SDP}}$ described in Section VI.A, and computed by solving (\ref{BoydSDP})-(\ref{Constr}), is an SDP relaxation of the true MVOE; see e.g., \cite[Ch. 3.7.4.1]{BentalNemirovskiBook2001}. In the following, we will compare the Hausdorff distance between $\mathcal{X}(t)$ and its approximating outer ellipsoid $\widehat{\mathcal{E}}(t)$, where $\widehat{\mathcal{E}}(t)$ is either the MTOE $\mathcal{E}_{\text{MTOE}}(t)$, or one of the MVOEs: $\mathcal{E}_{\text{MVOE}}^{\text{SDP}}(t)$ and $\mathcal{E}_{\text{MVOE}}^{\text{proposed}}(t)$.
\end{remark}

For the setup considered in (\ref{LinearSys}) and (\ref{pSumInitCondnControl}), the reach set $\mathcal{X}(t)$ is guaranteed to be convex, which allows us to transcribe (\ref{HausdorffDefn}) in terms of the support functions:
\begin{eqnarray}
\delta_{{\rm{H}}}(t) = \sup_{\bm{s}\in\mathcal{S}^{n_{x}-1}} \big\vert h_{\widehat{\mathcal{E}}(t)}(\bm{s}) \:-\: h_{\mathcal{X}(t)}(\bm{s}) \big\vert,
\label{HausdorffDistSptFn}	
\end{eqnarray}
where $\mathcal{S}^{n_{x}-1}$ denotes the Euclidean unit sphere embedded in $\mathbb{R}^{n_{x}}$. Thanks to property (iv) in Section II.1, the absolute value in (\ref{HausdorffDistSptFn}) can be dropped. Furthermore, suppose that both $\mathcal{X}(t)$ and $\widehat{\mathcal{E}}(t)$ are centered at origin, as in Section VI; in particular, $\widehat{\mathcal{E}}(t) \equiv \mathcal{E}\left(\bm{0},\widehat{\bm{Q}}(t)\right)$. Using (\ref{FireypSum}), (\ref{SptFnEllipsoid}), (\ref{pSumInitCondnControl}) and (\ref{dLTVSptFn}), we can then rewrite (\ref{HausdorffDistSptFn}) as
\begin{align}
&\delta_{{\rm{H}}}(t) \nonumber\\
&=\!\sup_{\bm{s}\in\mathcal{S}^{n_{x}-1}}\!\!\left(\bm{s}^{\top}\!\widehat{\bm{Q}}(t)\bm{s}\right)^{\!\frac{1}{2}} - \Bigg\{\!\!\left(\!\displaystyle\sum_{i=1}^{m}\!\left(\!\bm{s}^{\top}\bm{F}^{t}\bm{Q}_{0i}\left(\bm{F}^{\top}\right)^{\!t}\!\bm{s}\!\right)^{\!\!\frac{p_{1}}{2}}\!\right)^{\!\!\frac{1}{p_{1}}}\nonumber\\
&+ \displaystyle\sum_{k=0}^{t-1}\left(\displaystyle\sum_{j=1}^{n}\!\left(\bm{s}^{\top}\bm{F}^{t-k-1}\bm{G}\bm{U}_{j}(t)\bm{G}^{\top}\left(\bm{F}^{\top}\right)^{\!t-k-1}\bm{s}\right)^{\!\!\frac{p_{2}}{2}}\right)^{\!\!\frac{1}{p_{2}}}\!\!\Bigg\}.
\label{HausdorffMVOEvspSum}	
\end{align}
In words, (\ref{HausdorffMVOEvspSum}) is the two-sided Hausdorff distance between the reach set of (\ref{LinearSys}) and its outer ellipsoidal approximation at time $t$, provided $\mathcal{X}_{0}$ is the $p_{1}$-sum of $m$ centered ellipsoids, and $\mathcal{U}(t)$ is the $p_{2}$-sum of $n$ time-varying centered ellipsoids, where  $p_{1},p_{2}\geq 1$.

To illustrate the use of (\ref{HausdorffMVOEvspSum}), let us consider $m=n=1$ as in Section VI.A. In this case, (\ref{HausdorffMVOEvspSum}) can be expressed succinctly as
\begin{eqnarray}
\delta_{{\rm{H}}}(t) = \sup_{\bm{s}\in\mathcal{S}^{n_{x}-1}}\! \big\Vert \widehat{\bm{Q}}^{\frac{1}{2}}(t)\bm{s} \big\Vert_{2} - \displaystyle\sum_{k=0}^{t} \big\Vert\bm{M}_{k}^{\frac{1}{2}}(t)\bm{s}\big\Vert_{2},
\label{HausdorffMVOEvsMinkSum}	
\end{eqnarray}
where $\bm{M}_{k} := \bm{F}^{t-k-1}\bm{G}\bm{U}(t)\bm{G}^{\top}\left(\bm{F}^{\top}\right)^{\!t-k-1} \in \mathbb{S}_{+}^{n_{x}}$ for $k=0,1,\hdots, t-1$, and $\bm{M}_{t} := \bm{F}^{t}\bm{Q}_{0}\left(\bm{F}^{\top}\right)^{\!t} \in \mathbb{S}_{+}^{n_{x}}$. From (\ref{HausdorffMVOEvsMinkSum}), a simple upper bound for $\delta_{{\rm{H}}}(t)$ follows. 

\begin{proposition}\label{HausdorffUBprop} (\textbf{Upper bound for} $\delta_{{\rm{H}}}(t)$ \textbf{when} $m=n=1$)
\begin{eqnarray}
	\delta_{{\rm{H}}}(t) \leq \bigg\Vert \widehat{\bm{Q}}^{\frac{1}{2}}(t) - \displaystyle\sum_{k=0}^{t} \bm{M}_{k}^{\frac{1}{2}}(t) \bigg\Vert_{2}.
\label{dHub}	
\end{eqnarray}	
\end{proposition}
\begin{proof}
For $\bm{s}\in\mathcal{S}^{n_{x}-1}$	, by the (repeated use of) reverse triangle inequality, we have
\begin{align*}
\big\Vert \widehat{\bm{Q}}^{\frac{1}{2}}(t)\bm{s} \big\Vert_{2} \!-\! \displaystyle\sum_{k=0}^{t}\!\big\Vert\bm{M}_{k}^{\frac{1}{2}}(t)\bm{s}\big\Vert_{2} &\leq \Bigg\Vert \left(\!\widehat{\bm{Q}}^{\frac{1}{2}}(t) - \displaystyle\sum_{k=0}^{t} \bm{M}_{k}^{\frac{1}{2}}(t)\!\right)\bm{s} \Bigg\Vert_{2}\\
&\leq \bigg\Vert \widehat{\bm{Q}}^{\frac{1}{2}}(t) - \displaystyle\sum_{k=0}^{t} \bm{M}_{k}^{\frac{1}{2}}(t) \bigg\Vert_{2}, 	
\end{align*}
where the last step follows from the sub-multiplicative property of the matrix 2-norm. Since this holds for any $\bm{s}\in\mathcal{S}^{n_{x}-1}$, the same holds for the optimal $\bm{s}$ in (\ref{HausdorffMVOEvsMinkSum}). Hence the result.
\end{proof} 
\begin{figure}[t]
 \centering
\includegraphics[width=\linewidth]{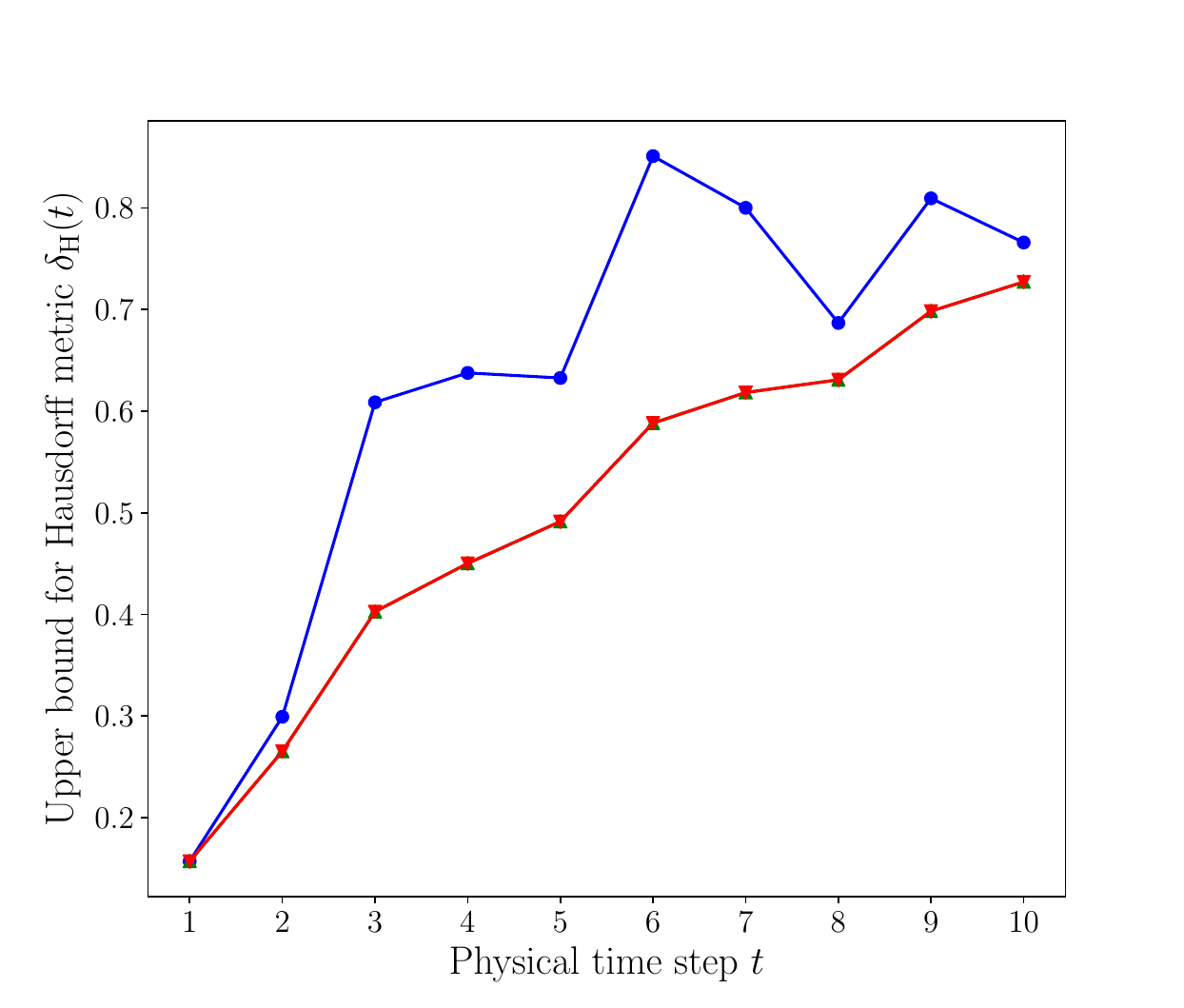}
 \caption{The upper bound (\ref{dHub}) for the Hausdorff distance $\delta_{\rm{H}}(t)$ between the reach set of (\ref{LinearSys}) subject to ellipsoidal uncertainties (the case $m=n=1$), and its  various MVOE approximations. The data for this plot are same as in Section \ref{ReachabilitySubsection}, and correspond to Fig. \ref{ReachSetPlot}, \ref{FPCompTime} and Table \ref{TableVolComparison}. The MVOE from the SDP computation (\protect\circleline) results in a larger upper bound than the MVOEs computed using the fixed point recursion (\protect\invtriangleline) proposed herein, and the root-bracketing algorithm (\protect\triangleline) proposed in \cite{HalderMinkSum2018}.}
 \label{HausdorffPlot}
\end{figure}
We can use the bound derived in Proposition \ref{HausdorffUBprop} as a conservative numerical estimate for $\delta_{{\rm{H}}}(t)$. In Fig. \ref{HausdorffPlot}, we use the data from Section \ref{ReachabilitySubsection} to plot the upper bound (\ref{dHub}) for the Hausdorff distance between the reach set of (\ref{LinearSys}) and its MVOE approximations. The plot shows that the MVOEs computed using the proposed algorithms result in a lower upper bound than the MVOE computed using the standard SDP approach. This indicates that the proposed MVOEs approximate the reach set quite well compared to the SDP approach, even though their volumes are slightly larger than the SDP MVOEs, as noted in Table \ref{TableVolComparison}.

\subsection{Computational Complexity}\label{subsecComputationalComplexity}
The runtime for computing the approximate MVOE for the $p$-sum of a given number of summand ellipsoids using the proposed fixed point recursion (\ref{FPE}) scales as $\mathcal{O}\left(d^{3}\right)$ when each of the (non-degenerate) summand ellipsoids are in $\mathbb{R}^{d}$. This is because the fixed point recursion needs the generalized eigenvalues $\{\lambda_{i}\}_{i=1}^{d}$ of the matrix pencil formed by a pair of shape matrices associated with the summand ellipsoid pair -- a task that takes $\mathcal{O}\left(d^{3}\right)$ worst-case runtime, and dominates the remaining $\mathcal{O}(d)$ computation. In particular, the computation at time $t$ in Section \ref{ReachabilitySubsection} requires finding the approximate MVOE of the Minkowski sum of $t+1$ summand ellipsoids in $\mathbb{R}^{n_{x}}$, resulting in $\mathcal{O}\left(t n_{x}^{3}\right)$ runtime. Likewise, the computation at time $t$ in Section \ref{subsecmnmorethan1} requires finding the approximate MVOE of the Minkowski sum of $t+1$ summand convex sets in $\mathbb{R}^{n_{x}}$, where one of the summand is the $p_{1}$-sum of $m$ ellipsoids, and each of the remaining $t$ convex sets are the $p_{2}$-sum of $n$ ellipsoids, thus resulting in $\mathcal{O}\left((m + (t-1)n)n_{x}^{3}\right)$ runtime. This can be interpreted as follows. Since $m$ and $n$ encode the complexity of the set-valued initial condition and control uncertainties, the worst-case runtime scales linearly with the complexity of the set-valued uncertainty description, and cubic in state dimension.


\section{Conclusions}
Computing a tight ellipsoidal outer approximation of convex set-valued data is necessary in many systems-control problems such as reachability analysis, especially when the convex set is described as set operations on other ellipsoids. Depending on the set operation, isolated results and algorithms are known in the literature to compute the minimum trace and minimum volume outer ellipsoidal approximations. In this paper, we unify such results by considering the $p$-sum of ellipsoids, which for different $p\in[1,\infty]$, generate different convex sets from the summand ellipsoids. Our analytical results lead to efficient numerical algorithms, which are illustrated by the reach set computation for a discrete-time linear control system. A specific direction of future study is to extend the reach set computation for hybrid systems, by computing the intersection of the guard sets with $p$-sum of ellipsoids. Comparing the proposed algorithms with existing reachability computation tools on benchmark problems will be pursued in our follow-up work. We hope that the models and methodologies presented here, will help to expand the ellipsoidal calculus tools \cite{KurzhanskiValyi1997,schweppe1973,KuzhanskiVaraiyaBook2014,kurzhanskiy2007ellipsoidal,gagarinov2012computation,ros2002ellipsoidal} for systems-control applications, and motivate further studies to leverage the reported computational benefits in real-time applications.


\section*{Acknowledgement}
The author is grateful to Suvrit Sra for suggesting \cite{MOsuvrit} the fixed point iteration (\ref{FPE}) for $p=1$. This research was partially supported by a 2018 Faculty Research Grant awarded by the
Committee on Research from the University of California, Santa Cruz, and
by a 2018 Seed Fund Award from CITRIS and the Banatao Institute at the University of California. The author is indebted to Lia Gianfortone for discussions and help with preliminary numerical investigation.


\appendix

\subsection{Set Inclusion for the $p$-Sum}\label{AppendixSetInclusionpSum}
In the following, we formally state and prove the set
inclusion relation for the $p$-sum.

\begin{proposition}
Given compact convex sets $\mathcal{K}_{1},\mathcal{K}_{2}$, we have $\mathcal{K}_{1} +_{q} \mathcal{K}_{2} \subseteq \mathcal{K}_{1} +_{p} \mathcal{K}_{2}$	for $1 \leq p < q < \infty$.
\end{proposition}
\begin{proof}
We proceed in two steps: \emph{first}, we show that $h_{\mathcal{K}_{1} +_{q} \mathcal{K}_{2}} \leq h_{\mathcal{K}_{1} +_{p} \mathcal{K}_{2}}$, $1 \leq p < q < \infty$ holds; \emph{second}, we show that for any two compact convex sets $\mathcal{A},\mathcal{B}$, the support function inequality $h_{\mathcal{A}} \leq h_{\mathcal{B}}$ is equivalent to $\mathcal{A} \subseteq \mathcal{B}$.	

From Definition \ref{pSumDefn}, showing the \emph{first step} amounts to proving $\left(\left(h_{\mathcal{K}_{1}}(\bm{x})\right)^{q} + \left(h_{\mathcal{K}_{2}}(\bm{x})\right)^{q}\right)^{1/q} \leq \left(\left(h_{\mathcal{K}_{1}}(\bm{x})\right)^{p} + \left(h_{\mathcal{K}_{2}}(\bm{x})\right)^{p}\right)^{1/p}$ for $1 \leq p < q < \infty$, which clearly holds from monotonicity. For the \emph{second step}, we recall that the Legendre-Fenchel conjugate of a function $f(\bm{x})$, denoted as $f^{*}(\bm{y})$, is given by
\[f^{*}(\bm{y}) := \underset{\bm{x}}{\sup}\:\langle\bm{y},\bm{x}\rangle - f(\bm{x}).\]
As a consequence of the above definition, $f \geq g \Rightarrow f^{*} \leq g^{*}$. Furthermore, recall the well-known fact that for a proper, closed, convex function $f$, its bi-conjugate $f^{**} = f$. Now, notice that for any two compact convex sets $\mathcal{A},\mathcal{B}$, the support functions $h_{\mathcal{A}}, h_{\mathcal{B}}$ are the conjugates of the respective set indicator functions $\mathbf{1}_{\mathcal{A}}, \mathbf{1}_{\mathcal{B}}$, i.e., $h_{\mathcal{A}} = \mathbf{1}_{\mathcal{A}}^{*}$, and $h_{\mathcal{B}} = \mathbf{1}_{\mathcal{B}}^{*}$. Here,
\[\mathbf{1}_{\mathcal{A}}(\bm{x}) := \begin{cases}
 0 & \text{if}\;\bm{x}\in\mathcal{A},\\
 +\infty & \text{otherwise}.	
 \end{cases}
\]
The set indicator function $\mathbf{1}_{\mathcal{B}}(\bm{x})$ is defined likewise. Thus,
\[\mathcal{A}\subseteq\mathcal{B} \:\Rightarrow\: \mathbf{1}_{\mathcal{A}} \geq \mathbf{1}_{\mathcal{B}} \:\Rightarrow\: \mathbf{1}_{\mathcal{A}}^{*} \leq \mathbf{1}_{\mathcal{B}}^{*} \:\Rightarrow\: h_{\mathcal{A}} \leq h_{\mathcal{B}}.\]
Conversely,
\[h_{\mathcal{A}} \leq h_{\mathcal{B}} \:\Rightarrow\: h_{\mathcal{A}}^{*} \geq h_{\mathcal{B}}^{*} \:\Rightarrow\: \mathbf{1}_{\mathcal{A}} \geq \mathbf{1}_{\mathcal{B}} \:\Rightarrow\: \mathcal{A}\subseteq\mathcal{B}.\]
This proves the second step, i.e., $\mathcal{A} \subseteq \mathcal{B} \Leftrightarrow h_{\mathcal{A}} \leq h_{\mathcal{B}}$. We conclude the proof by setting $\mathcal{A}\equiv \mathcal{K}_{1} +_{q} \mathcal{K}_{2}$ and $\mathcal{B} \equiv \mathcal{K}_{1} +_{p} \mathcal{K}_{2}$ for $1 \leq p < q < \infty$, and then combining the two steps.
\end{proof}

\begin{figure}[t]
 \centering
\includegraphics[width=\linewidth]{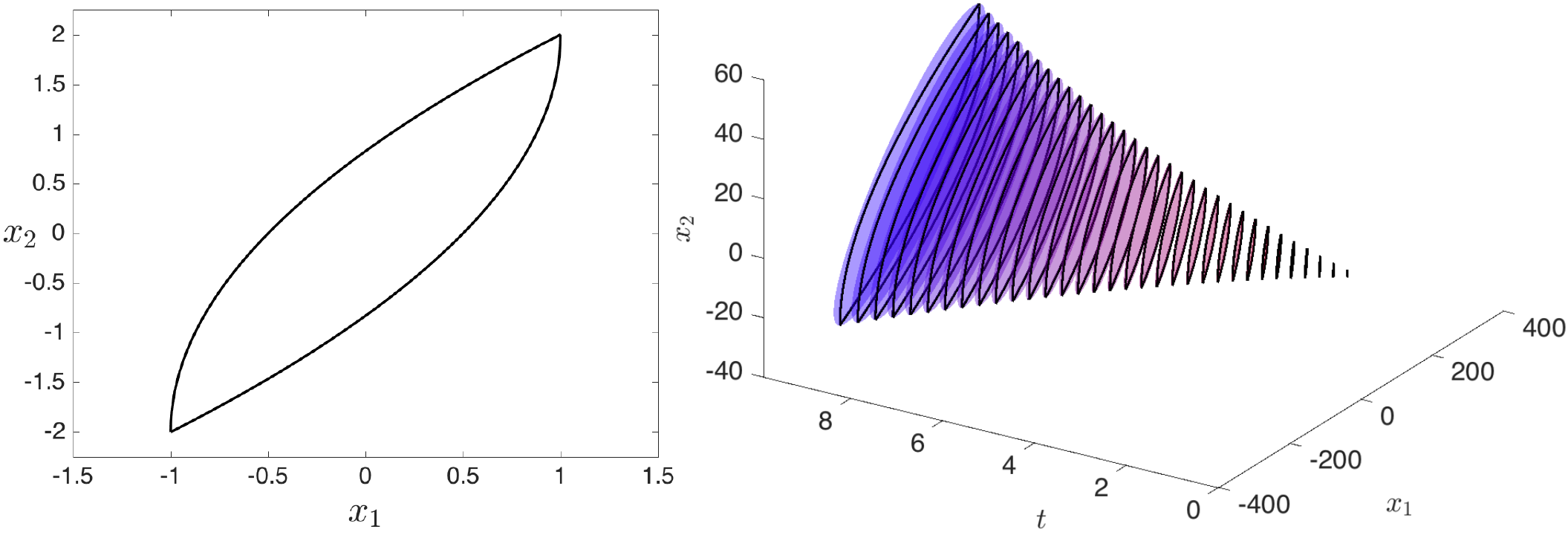}
 \caption{\emph{Left:} The reach set for the double integrator $\ddot{x}=u$ at $t=1$ with initial condition $\left(x_{10},x_{20}\right)\equiv(0,0)$, and control bound: $|u| \leq 2$. \emph{Right:} 33 different snapshots of the reach sets (black solid outlines) for the double integrator with $|u| \leq 2$, $t\in[0,10]$, superimposed with their respective approximate MVOEs (filled, no outline) at those times. For the MVOE computation, the initial set-valued ellipsoidal uncertainty is taken as $\mathcal{E}\left(\bm{0},10^{-4}\bm{I}_{2}\right)$, where $\bm{0}$ denotes the $2\times 1$ zero vector, and $\bm{I}_{2}$ is the $2\times 2$ identity matrix. This small ellipsoidal initial condition uncertainty serves as a proxy for the fixed initial condition $(0,0)$ used in plotting the true reach sets (solid outlines), and allows using the proposed fixed point algorithm.}
 \label{FigDoubleIntegratorReachSet}
\end{figure}

\subsection{Approximate MVOE for the Double Integrator}\label{AppendixDoubleIntegrator}
We performed the approximate MVOE computation for the reach set of the double integrator $\ddot{x} = u$ with bounded control $|u|\leq \mu$, $\mu > 0$, using the proposed Algorithm \ref{AlgoMVOE}. Fig. \ref{FigDoubleIntegratorReachSet} shows that the resulting approximate MVOEs obtained via the proposed algorithm provide tight outer-approximations of the true reach sets. This example is special in the sense that the true reach set starting from $\left(x_{10},x_{20}\right)\equiv(0,0)$ with the control bound $|u|\leq \mu$ can be analytically computed (see e.g., \cite[p. 111]{KuzhanskiVaraiyaBook2014}) as follows. Let the parameter $\sigma\in[-t,0]$. the parametric equation for the upper boundary $\left(x_{1}^{+}(t),x_{2}^{+}(t)\right)$ is
\begin{subequations}
\begin{align}
x_{1}^{+}(t) &= x_{10} + x_{20} t + \mu \left(\frac{t^{2}}{2} - \sigma^{2}\right),\\
x_{2}^{+}(t) &= x_{20} + 2\mu\sigma + \mu t.	
\end{align}
\label{UpperBndry}	
\end{subequations}
The same for the lower boundary $\left(x_{1}^{-}(t),x_{2}^{-}(t)\right)$ is
\begin{subequations}
\begin{align}
x_{1}^{-}(t) &= x_{10} + x_{20} t - \mu \left(\frac{t^{2}}{2} - \sigma^{2}\right),\\
x_{2}^{-}(t) &= x_{20} - 2\mu\sigma - \mu t.	
\end{align}
\label{LowerBndry}	
\end{subequations}
These analytical expressions not only help in visually con- firming the performance of the proposed algorithm, but also allow us to quantify the ratio $\vol\left(\mathcal{E}_{\text{MVOE}}(t)\right)/\vol\left(\mathcal{X}(t)\right)$, where the numerator equals $\pi\sqrt{\det\left(\bm{Q}(t)\right)}$, the matrix $\bm{Q}(t)$ being the $2\times 2$ shape matrix obtained from our MVOE algorithm at time $t$, and the denominator has closed-form expression $2\mu^{2}t^{3}/3$, where $|u| \leq \mu$. We found the said volume ratio to be $\geq 85\%$ for all 33 snapshots in Fig. \ref{FigDoubleIntegratorReachSet}. 

In general, the true reach set is not available, and one resorts to indirect methods (see Section \ref{subsecQualityOfApprox}) to assess the quality of the outer-approximation.

\begin{figure}[t]
 \centering
\includegraphics[width=\linewidth]{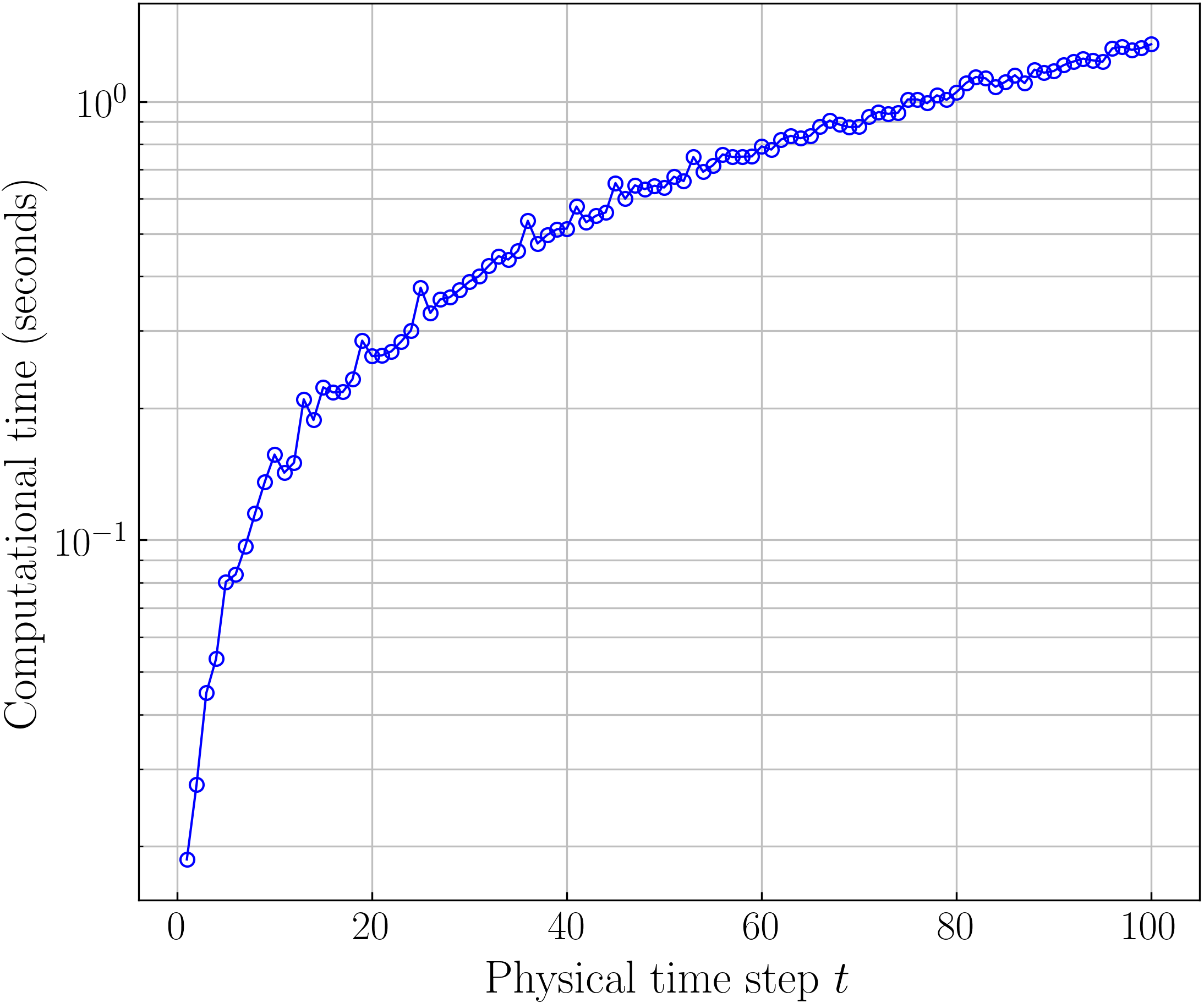}
 \caption{The computational times for the MVOE calculations reported in Appendix \ref{AppendixISS} using the proposed Algorithm 2.}
 \vspace*{-0.15in}
 \label{FigISScomptime}
\end{figure}

\subsection{Large Scale Example of Approximate MVOE Computation}\label{AppendixISS}
    To illustrate the scalability of the proposed algorithm, we now provide a large scale numerical example for MVOE computation. Specifically, we consider a continuous time linear time invariant (LTI) model of the first assembly stage of the International Space Station (ISS), also known as the Russian service module 1R or ISS-I model \cite{Gugercinc2001CDC,Chahlaoui2005,tran2016large}. The model consists of 270 states and 3 controls; the model data are available at \url{http://verivital.com/hyst/benchmark-large-scale/}. From this model data, we generate the corresponding discrete-time LTI system matrices $\bm{F}$ (of size $270\times 270$) and $\bm{G}$ (of size $270\times 3$) with sampling time $h=0.05$ seconds. As in Section VI.A, we consider $m=n=1$, for which the reach state $\mathcal{X}(t)\subset\mathbb{R}^{270}$ is of the form (\ref{EllipsoidalMinkSum}) for the ellipsoidal uncertainties in initial condition modeled via the shape matrix $\bm{Q}_{0}$, and the (time-varying) ellipsoidal actuation uncertainties modeled via the shape matrices $\bm{U}(t)$. We set $\bm{Q}_{0}=\bm{I}_{270}$, and $\bm{U}(t) = (1 + \cos^{2}(t)){\rm{diag}}\left({\rm{rand}}(3,1)\right)$, that is, time-varying positive scaling of random positive diagonal matrices. We used Algorithm 2 to compute the approximate MVOEs of $\mathcal{X}(t)$ for each $t=1,2,\hdots,100$. As in Section VI.A, this for any given $t$, amounts to computing the MVOE of the 1-sum of $t+1$ summand ellipsoids. The resulting computational times are shown in Fig. \ref{FigISScomptime}.


%
%

\end{document}